\documentclass[10pt,a4paper]{amsart}
\usepackage{latexsym}
\usepackage{amssymb} 
\usepackage{enumerate} 
\usepackage{ifpdf}
\usepackage{enumitem} 
 
\usepackage{verbatim} 
\usepackage{mathrsfs} 
\usepackage{multirow} 
\usepackage{graphicx} 
\usepackage{hyperref} 
 
\usepackage{tikz} 
\usetikzlibrary{matrix,arrows}

\newcommand\rightmap[1]{\smash{\mathop{\ \rightarrow\ }\limits^{#1}}} 
\newcommand\longrightmap[1]{\overset{#1}{\longrightarrow}}

\newcommand{\cC}{\mathcal C} 
\newcommand{\cX}{\mathcal X} 
\newcommand{\cA}{\mathcal A}

\newcommand{\cH}{\mathcal H}

\newcommand{\cO}{\mathcal O}

\newcommand{\cR}{\mathcal R} 
\newcommand{\cQ}{\mathcal Q}
\newcommand{\sV}{\mathring{V}}

\newcommand{\e}{\varepsilon}

\newcommand{\Alb}{\text{\rm Alb}} 
\newcommand{\Mor}{\text{\rm Mor}} 
 
\newcommand{\Gr}{\text{\rm Gr}} 
\newcommand{\Hom}{\text{\rm Hom}} 
\newcommand{\orb}{\text{\rm orb}}

\newcommand{\red}{\text{\rm red}} 
\newcommand{\coker}{\text{\rm coker}} 
\newcommand{\im}{\text{\rm Im}} 
\DeclareMathOperator{\Spec}{Spec} 
\DeclareMathOperator{\ab}{ab} 
\DeclareMathOperator{\Char}{Char} 
\DeclareMathOperator{\rk}{rank} 
\DeclareMathOperator{\lk}{lk} 
\DeclareMathOperator{\res}{Res}

\newcommand{\ZZ}{\mathbb Z} 
\newcommand{\CC}{\mathbb C}

\newcommand{\QQ}{\mathbb Q} 
\newcommand{\PP}{\mathbb P} 
 
\newcommand{\GG}{\mathbb G} 
\newcommand{\KK}{\mathbb K}

\def\cal{\mathcal} 
\def\SC{{Special Ceva}}
\newcommand{\Ceva}{\text{\rm CEVA}}
 
\newtheorem{theorem}{Theorem}[section] 
\newtheorem{lemma}[theorem]{Lemma} 
\newtheorem{prop}[theorem]{Proposition} 
\newtheorem{cor}[theorem]{Corollary} 
\theoremstyle{definition} 
\newtheorem{dfn}[theorem]{Definition}

\newtheorem{example}[theorem]{Example} 
 
\theoremstyle{remark} 
\newtheorem{remark}[theorem]{Remark} 
 
\numberwithin{equation}{section} 
 
\newcommand\enet[1]{\renewcommand\theenumi{#1} 
\renewcommand\labelenumi{\theenumi}}

\begin{document} 
 
\title[Depth of cohomology support loci...] 
{Depth of cohomology support loci for quasi-projective varieties via orbifold pencils} 
\author{E. Artal Bartolo, J.I.~Cogolludo-Agust{\'\i}n and A.~Libgober} 

\address{Departamento de Matem\'aticas, IUMA\\ 
Universidad de Zaragoza\\ 
C.~Pedro Cerbuna 12\\ 
50009 Zaragoza, Spain} 
\email{artal@unizar.es,jicogo@unizar.es}
 
\address{ 
Department of Mathematics\\ 
University of Illinois\\ 
851 S.~Morgan Str.\\ 
Chicago, IL 60607} 
\email{libgober@math.uic.edu} 
 
\thanks{Partially supported by the Spanish Ministry of
Education MTM2010-21740-C02-02. 
The third author was partially supported by NSF grant.} 
 

 
\begin{abstract} The present paper describes a relation between 
the quotient of the fundamental group of a smooth quasi-projective variety 
by its second commutator and the existence of maps to orbifold curves. 
It extends previously studied cases when the target was a smooth curve. 
In the case when the quasi-projective variety is a complement
to a plane algebraic curve this provides new relations between 
the fundamental group, the equation of the curve, and the existence of 
polynomial solutions to certain equations generalizing Pell's equation.
These relations are formulated in terms of the depth which 
is an invariant of the characters of the 
fundamental group discussed in detail here.
\end{abstract}

\maketitle


\section{Introduction}

Let $\cal X$ be a smooth quasi-projective variety and let $\chi\! \in\! \Hom(\pi_1(\cal X),\!\CC^*)$ 
be a character of its fundamental group. Viewing $\chi$ as a rank one local system, one associates 
to it the twisted cohomology groups. The purpose of this note is to extend known relations between 
holomorphic maps of $\cal X$ onto curves, i.e. holomorphic pencils, and dimensions of the twisted 
cohomology $H^1(\cal X,\chi).$

The problem of the existence of holomorphic pencils can be traced back to almost one hundred years and in its
projective version\footnote{and where local systems are replaced by holomorphic bundles.} goes 
back to Castelnuovo, deFrancis, Catanese, Green-Lazarsfeld, and Simpson (cf.~\cite{manuscripta} 
for a list of references). The quasi-projective case was considered in~\cite{arapura}, where the 
structure of the jumping subsets of the variety of characters 
\begin{equation}\label{jump}
\sV_k(\cX):=\{\chi\in\Hom(\pi_1(\cal X),\CC^*) \mid \dim H^1(\cal X,\chi) = k \} 
\end{equation}
was studied together with its relation to pencils. In this context, if $\chi\in \sV_k(\cX)$ 
we say $\chi$ has \emph{depth}~$k$. 
The characteristic varieties~$V_k(\cX)$ are defined analogously to $\sV_k(\cX)$, but replacing~$=$ by 
$\geq$ in~\eqref{jump}. This term was introduced in~\cite{charvar} for complements to 
plane curves and explicitly related to the structure of the fundamental group in~\cite{eko,charvar}.
To be more precise, the characteristic varieties referred to above can be described as the zero sets 
of the Fitting ideals of the abelianization $\pi_1'/\pi_1''$ of the commutator of~$\pi_1$, which 
coincide with the jumping loci~\eqref{jump} outside of the trivial character 
(see Theorem~\ref{otherdepthdef}). In particular, the characteristic varieties (unlike the jumping 
sets for the higher cohomology spaces) depend only on the fundamental group. Fox calculus provides 
an effective method for calculating the characteristic varieties in the cases when a presentation of 
the fundamental group by generators and relators is known. 

The results of~\cite{arapura} are as follows:
each $V_k(\cX)\subset \Hom(\pi_1(\cal X),\CC^*)$ 
is a finite union of translated subgroups
(i.e. cosets) of $\Hom(\pi_1(\cX),\CC^*)$. 
Moreover, for each component of positive dimension
there exists a curve $C$ with 
negative Euler characteristic such that this component has the form
$\rho \cdot f^*\Hom(\pi_1(C),\CC^*)$ for some holomorphic map 
$f: \cal X \rightarrow C$. This was supplemented in~\cite{manuscripta} 
by showing that the zero-dimensional components have finite order.
 
 
A more precise version of this result can be found in~\cite{acm-arapura} in
terms of orbifolds. It includes some missing points regarding resonant
conditions and extends the result from
$V_1(\cX)$ to all characteristic varieties $V_k(\cX)$, $k\geq 1$. 

If $\cal X$ is a complement to a plane projective curve,
the target $C$ of a holomorphic pencil mentioned 
above must be necessarily $C=\PP^1\setminus \{\text{points}\}$
and thus $f$ extends to a rational 
pencil on $\PP^2$. 
In this case, positive dimensional
translated components $\rho \cdot f^*\Hom(\pi_1(C),\CC^*)$ of $V_1(C)$ have been shown
(see Dimca~\cite{dimca-pencils}) to be related to the multiple fibers of such a pencil,
see also~\cite{acm-arapura}. 

For a generic non-isolated character $\chi \in V_k(C)$ in a component of $V_k(C)$ of dimension
greater than one, the following formula for its depth holds: 
\begin{equation}
\dim H^1(C,\chi)=
\begin{cases}
\dim V_k(C)-2=-e(C) & \text{ if } C \text{ is compact} \\
\dim V_k(C)-1=-e(C) & \text{ otherwise,}
\end{cases}
\end{equation}
where $e(C)$ is the (topological) 
Euler characteristic of $C$.\footnote{For example,
if $C=\PP^1\setminus \{n \text{ points}\}$, then 
$\dim(H^1(C,\CC^*))=n-1$ and $\dim(H^1(C,\chi))=-e(C)=n-2$,
if $\chi$ is non-trivial.} 

This provides a simple way to determine or at least to estimate
the depth of characters on components having a {\it positive} 
dimension. 

Isolated points in components $V_k(\cX)$ are common occurrence and below we describe 
the geometric significance of the depth of zero dimensional irreducible 
components of $V_k(\cX)$. We do so using {\it orbifold pencils} associated with such 
characters (as was mentioned, such characters must have a finite order).
 
It is worth mentioning that the nature of the cohomology 
of local systems is essentially different depending on whether
$H_1(\bar{\cal X},\CC)$ is trivial or not. In the latter case, 
due to the surjection
$\pi_1(\cal X) \rightarrow \pi_1(\bar{\cal X}) \rightarrow 1$, 
some of the characters
of $\pi_1(\cal X)$ are the characters of the projective fundamental group
(cf.~\cite{simpson} for a discussion on the difference between the projective and the
quasi-projective case). In this paper (as in~\cite{arapura})
we shall focus on the case when $H_1(\bar{\cal X},\CC)=0$. 
This includes the case of the complements to plane curves which provides many 
concrete and interesting examples.

For the basics on the theory of orbifolds we refer to~\cite{adem} or, since we shall
consider mainly orbifold curves, to~\cite{scott} or~\cite{friedman}. An orbifold
pencil is a (birational) dominant map $\cal X \rightarrow \cal C$, where $\cal C$
is the orbicurve such that the preimage of each point in $\cal C$ with stabilizer of
order $m$ is a multiple fiber of order a multiple of~$m$.
A proof of this lemma can be found in~\cite{acm-arapura}. 
 
\begin{lemma}
\label{lemma-ofg}
An orbifold pencil $f:\cal X \rightarrow \cal C$ defines a morphism of orbifold 
fundamental groups $f_*:\pi_1(\cal X) \rightarrow \pi_1^\orb(\cal C)$. 
\end{lemma}
 

We call (cf. Definition~\ref{def-orb-marking}) the map described in the 
previous lemma a \emph{marked orbifold pencil} $f: \cX \rightarrow \cC$.
The \emph{markings} are given by the pairs $(\cX,\chi), \chi \in V_k(\cX)$,
and $(\cC,\rho), \rho \in V_k^{\orb}(\cC)$ such that $f^*(\rho)=\chi$, where 
$f^*$ is the map of groups of characters corresponding to~$f_*$. Note that 
$V_k^{\orb}(\cC)$ is the orbifold characteristic variety of $\cC$ defined 
as $V_k$ in~\eqref{jump} for $\pi_1^{\orb}({\cal C})$,
which only depends on the group, as mentioned above.

A pair $(\cX,\chi)$ can be marked by several orbifold pencils
and we show that the number of such markings is related
in an appropriate sense to the depth (cf. Theorem~\ref{mordweiltheorem} 
below and section~\ref{quasiprojpencils-sec-thm}).

%
 
The relation between orbifold pencils and local systems with non-vanishing cohomology 
was studied in~\cite{mordweil}. 
In this paper, the problem of finding a bound on the degree of the 
Alexander polynomial is discussed for plane curves with cusps and nodes as the only
singularities (or curves with singularities in a more general class of $\delta$-essential singularities).
The connection with the cohomology of local systems is 
coming from the following:
For an irreducible curve $D$ in $\PP^2$
one has $H_1(\PP^2\setminus D,\ZZ)=\ZZ/{\deg D}\ZZ$ 
i.e. $\Hom(\pi_1(\cal X),\CC^*)=\mu_{\deg D}$ and 
if $\chi_{\xi}$ corresponds to $\xi \in \mu_{\deg D}$ 
then:
\begin{equation}
\Delta(\xi)=0 \Rightarrow\dim H^1(\PP^2\setminus D,\chi_{\xi})\ne 0. 
\end{equation}
The key step in \cite{mordweil}
for obtaining the bound on the degree of the Alexander polynomial
(or equivalently the multiplicity of the root $\exp(\frac{2 \pi i}{6})$)
was to show the following:

\begin{theorem}\label{mordweiltheorem}
The degree of the Alexander polynomial of a curve $D$ having 
cusps and nodes as the only singularities coincides with the number of 
independent orbifold pencils $\PP^2 \rightarrow \PP^1_{2,3,6}$ such that $D$ is the preimage of the
orbifold point having the cyclic group of order six as the 
stabilizer. This number of independent pencils equals the rank of the group of quasitoric relations 
\begin{equation}
u^2+v^3=w^6F,
\end{equation} 
where $F=0$ is a defining equation for $D$.
\end{theorem}
 
Theorem~\ref{mordweiltheorem} can also be extended to general Alexander polynomials and non-reduced curves. 

One of the main results of this paper is the following theorem
(proven in section~\ref{quasiprojpencils-sec-thm}) 
providing
the relations between orbifold pencils and depth. 
It shows that the number of independent pencils (with a given target) provides
a lower bound for the depth of a character. Moreover, for
an interesting class of characters this bound is exact. 
 
\begin{theorem}\label{thm-main} 
Let $\cal X$ be a quasi-projective manifold together with a character $\chi$.
\begin{enumerate}
\enet{\rm(\arabic{enumi})} 
\item\label{thm-main-part1}
Assume that there are $n$ strongly independent marked orbifold pencils with a fixed target $(\cC,\rho)$ and let $d(\rho)$ denote the 
depth of the character $\rho$ of $\pi_1^{orb}(\cC)$.
Then 
$d(\chi) \ge n d(\rho)$. 

\item\label{thm-main-part2}
If in addition
$\chi$ is a 2-torsion character and two is its only weight
(cf.~{\rm \ref{weightcharacter}} for a definition of weights of a character), then there are exactly
$d(\chi)$ strongly independent orbifold pencils on $\cX$ whose target is the
global $\ZZ_2$-orbifold $\cC=\CC_{2,2}$. These pencils are marked with the non-trivial character
$\rho$ of $\pi_1^{\orb}(\CC_{2,2})$ characterized by the condition that it extends to $\PP^1_{2,2}$. 
\end{enumerate} 
 
Moreover, if $\chi$ is a $d$-torsion character, then~\ref{thm-main-part1} implies that
$\phi_d(t)^{nd(\rho)}|\Delta_{X,\chi}(t)$ and \ref{thm-main-part2} implies that $d(\chi)$ is the multiplicity
of $\phi_2(t)=(t+1)$ as a factor of $\Delta_{X,\chi}(t)$, where $\phi_k(t)$ denotes the cyclotomic polynomial
of order~$k$.
\end{theorem}
 
See sections~\ref{sec-orbicurves}, \ref{quasiprojpencils}, and~\ref{essentialsection} for the required definitions.
The Hodge theoretical condition on $\chi$ of having two as its only weight,
can be characterized as the requirement of the equality of
the first Betti numbers of both the double cover of ${\mathcal{X}}$
defined by $\chi$
and its smooth compactification; see Theorem~\ref{thm-hodge} for
another characterization. We specialize these results to the case of complements to plane curves in section~\ref{curvessection}.
The group $\pi_1(\PP^2\setminus D)$ is closely related to its central extension
$\pi_1(\CC^2\setminus D)$ where $\CC^2$ is obtained from $\PP^2$ 
by deleting a generic line at infinity.
In this case the group of characters is $(\CC^*)^r$ where $r$ is the number of irreducible components.
The properties of the characters lying on coordinate components are essentially different and in
section~\ref{essentialsection} we give a Hodge theoretical characterization of coordinate essential characters.
 
In the case of plane curves the orbifold pencils correspond to solutions of certain equations over the function field $\CC(x,y)$.
For example, as mentioned in Theorem~\ref{mordweiltheorem}, 
the depth of characters of order $6$ is related to the number of independent polynomial solutions in $u,v,w$ of the
quasitoric equation $u^2+v^3=w^6F$ 
of type~$(2,3,6)$.
This also can be used to relate 
the cohomology of the Milnor fiber of arrangements of lines with triple 
points and solutions to the Catalan equation (cf.~\cite{catalan}). 

A similar result for characters of order two is shown in section~\ref{sec-qt}.
Let $D\subset \PP^2$ be a projective plane curve, $X_D:=\PP^2\setminus D$ its complement, and $\chi$ a 2-torsion 
character on $\pi_1(X_D)$. Denote by $Q_{(D,\chi)}$ the set of $(2,2,0)$-quasitoric relations associated to~$\chi$, that is, 
\begin{equation}
\label{eq-qtintro}
Q_{(D,\chi)}:=\{(f,g,h,U,V)\in \CC[x,y,z]^5 \mid f U^2-g V^2=h, f\cdot g=F, h_{\text{red}}|H\}/\sim ,
\end{equation}
where $D:=\{FH=0\}$, $\chi$ is ramified exactly along $F=0$ (see Definition~\ref{dfn-qt220}) and $\sim$ is the appropriate equivalence relation. Then 
 
\begin{theorem} 
\label{thm-qt220} 
The set of $(2,2,0)$-quasitoric relations $Q_{(D,\chi)}$ has a structure of 
a finitely generated abelian group and 
$$\rk Q_{(D,\chi)} \leq d(\chi).$$ 
Moreover, if $X_D$ and $\chi$ satisfy the conditions of Theorem{\rm~\ref{thm-main}\ref{thm-main-part2}}, then 
$$\rk Q_{(D,\chi)} = d(\chi).$$ 
\end{theorem} 
 
We refer to section~\ref{sec-qt} for the exact definition of $Q_{(D,\chi)}$ 
and the equivalence $\sim$ 
among quasitoric relations. 
This result is illustrated both in section~\ref{sec-examples} of this paper and in~\cite{pisa}
with several non-trivial examples aiming to describe a calculation method for the group structure of $Q_{(D,\chi)}$.

Finally, we note that there is a surprising connection between the polynomial
 equations considered in~\eqref{eq-qtintro} and
the Pell equations over the field of rational functions $\CC(x,y)$. 
Investigations of the Pell equations 
\begin{equation}\label{pell}
u^2-f(x)v^2=1
\end{equation}
over the function field $\CC(x)$
apparently go back to Abel~\cite{abel} (cf. \cite{shabat-zvonkin}).
More recently, the equation (\ref{pell}) 
over $k[x]$ was considered by F.Hazama in~\cite{hazama-pell},
\cite{hazama-twists}
where a group structure closely 
resembling the one described in Theorem~\ref{thm-qt220} also appeared. 
A more detailed study of this connection is out of the scope
 of this note, but will appear elsewhere.

\section{Preliminaries}

\subsection{Characteristic varieties} \label{sec-charvar} 
\mbox{} 
 
Recall the basic definitions and results related to characteristic varieties and homology of covering
spaces. We will follow the original exposition given in~\cite{charvar}, but rephrase it in a more general setting. 
 
Throughout this section $X$ will be considered a topological space of finite type (that is, $X$ has the homotopy type
of a finite $CW$-complex), $\pi_1'(X)\subset \pi_1(X)$ be the commutator of the fundamental group. We shall
assume that $\pi_1(X)/\pi_1(X)' = H_1(X,\ZZ)$ is a free abelian group of rank $r$.
Basic examples are the complement to plane algebraic curves in $\CC^2$ with $r$ components and
links in a 3-sphere with $r$ components. 
 
Consider the torus of characters of $\pi_1(X)$ i.e. 
\begin{equation}
\Char(X):=\Hom(\pi_1(X),\CC^*).
\end{equation}
Alternatively, since $\Char(X)$ depends only on $\pi_1(X)$ we refer to it as $\Char(\pi_1(X))$.
This torus is canonically isomorphic to the spectrum $\Spec \CC[H_1(X,\ZZ)]=({\CC^*})^r$
of the group ring of abelianized $\pi_1(X)$.
Let $X_{\ab} \rightarrow X$ be the universal abelian cover i.e. the covering 
with the group $H_1(X,\ZZ)$. The group $H_1(X,\ZZ)$ acts
on $X_{\ab}$ as a group of automorphisms and this
provides $H_*(X_{\ab},\CC)$ with a structure of a $\CC[H_1(X,\ZZ)]$-module.
Recall that with each $R$-module~$M$ over a commutative
ring $R$ one associates the support which is the subvariety of $\Spec R$ consisting 
of the prime ideals~$\mathfrak{p}$ such that the localization~$M_{\mathfrak{p}}$ does not vanish. 

\begin{dfn}\label{charvardef} The \emph{characteristic variety} $V_k(X)$ is the
subvariety of the torus $\Char(X)=\Spec \CC[H_1(X,\ZZ)]$ given as the support of the module
$\bigwedge^k(H_1(X_{\ab},\CC))$ (the exterior power of the homology module). Alternatively, $V_k(X)$
can be given as the zero set of the $k$-th Fitting ideal of $H_1(X_{\ab},\CC)$, that is, the ideal generated by the 
$(n-k) \times (n-k)$ minors of the matrix $\Phi$ with coefficients in $\CC[\pi_1(X)]$ of the map $\Phi$:
\begin{equation}\label{presentation} 
\begin{tikzpicture}[description/.style={fill=white,inner sep=2pt},baseline=(current bounding box.center)] 
\matrix (m) [matrix of math nodes, row sep=3em, 
column sep=2.5em, text height=1.5ex, text depth=0.25ex] 
{ \CC[\pi_1(X)/\pi_1'(X)]^m& \CC[\pi_1(X)/\pi_1'(X)]^n& H_1(X_{\ab},\CC) &0.\\ }; 
\path[->,>=angle 90](m-1-1) edge node[auto] {$\Phi $} (m-1-2); 
\path[->,>=angle 90](m-1-2)edge node[auto,swap] {}(m-1-3); 
\path[->,>=angle 90](m-1-3)edge node[auto,swap] {}(m-1-4); 
\end{tikzpicture} 
\end{equation} 
We denote by $\sV_k(X)$ the set of the characters in $V_k(X)$ which do not belong to $V_j(X)$ for $j>k$.
If a character $\chi$ belongs to $\sV_k(X)$, then $k$ is called the  \emph{depth}
of $\chi$ and denoted by $d(\chi)$. 
\footnote{cf. Theorem~\ref{otherdepthdef} for comparison of this
definition and comments after~\eqref{jump}}
\end{dfn}
 
The following expresses the homology of finite abelian covers in terms of the depth of characters of $\pi_1$. 
The argument follows closely the one given in~\cite{abcovhomology} and~\cite{charvar}, but 
we will present some details here since the statement of Theorem~\ref{bettinumberofcover} is in a more
general context than in the references above. See also~\cite{acm-hefei,eko,sakuma}. 
 
\begin{theorem}\label{bettinumberofcover}
Let $X$ be a finite $CW$-complex, let $H$ be subgroup of $\pi_1(X)$ of
finite index containing the commutator $\pi_1(X)'$ and let $K:=\pi_1(X)/H$. 
 
Let $i_H: \Char(K) \rightarrow \Char(X)$ be the embedding of the character varieties induced 
by the surjection $\pi_1(X) \rightarrow K$. 
Let $X_H$ be the covering of $X$ corresponding to the subgroup $H$.

Then
\begin{equation}\label{formulabetti} 
b_1(X_H)=b_1(X)+\sum_{\xi \in \Char(K)\setminus\{1\}} d(i_H(\xi)).
\end{equation}
\end{theorem}

\begin{proof}
Consider
the five term
exact sequence corresponding to the spectral sequence: 
\begin{equation}
 E_2^{p,q}=H_p(L,H_q(X_{\ab},\CC)) \Longrightarrow H_{p+q}(X_H,\CC)
\end{equation}
which is the spectral sequence for the free action of the group $L:=H/\pi_1'(X)$ 
on the universal abelian cover $X_{\ab}$. It yields: 
\begin{equation}\label{fiveterms} 
\begin{tikzpicture}[description/.style={fill=white,inner sep=2pt},baseline=(current bounding box.center)] 
\matrix (m) [matrix of math nodes, row sep=3em, 
column sep=2.5em, text height=1.5ex, text depth=0.25ex] 
{ H_2(L,\CC)& H_1(X_{\ab})_{L} & H_1(X_H)&H_1(L,\CC) &0\\ }; 
\path[->,>=angle 90](m-1-1) edge node[auto] {} (m-1-2); 
\path[->,>=angle 90](m-1-2)edge node[auto,swap] {}(m-1-3); 
\path[->,>=angle 90](m-1-3)edge node[auto,swap] {}(m-1-4); 
\path[->,>=angle 90](m-1-4)edge node[auto,swap] {}(m-1-5); 
\end{tikzpicture} 
\end{equation} 
(subscript in second left term denotes covariants). Next, after taking 
the tensor 
product of sequence~\eqref{presentation} 
with the group ring $\CC[H/\pi_1'(X)]=\CC[L]$, using for a $\CC[L]$-module 
$M$ the identification of 
the covariants 
$M_{\CC[L]}$ with $M \otimes_{\CC[L]} \CC$ applied 
to the second term in~\ref{fiveterms} and finally using the isomorphism: 
\begin{equation} 
(\CC[\pi_1(X)/\pi_1'(X)])^s \otimes_{\CC[L]} \CC\!=\!\! 
(\CC[\pi_1(X)/\pi_1'(X)]/I_{\CC[L]})^s\!=\!\!(\CC[K])^s 
\end{equation}
(here $I_{\CC[L]}$ is the augmentation ideal) 
one obtains: 
\begin{equation}\label{quotientresolution} 
\begin{tikzpicture}[description/.style={fill=white,inner sep=2pt},baseline=(current bounding box.center)] 
\matrix (m) [matrix of math nodes, row sep=3em, 
column sep=2.5em, text height=1.5ex, text depth=0.25ex] 
{ \CC[K]^m& \CC[K]^n & H_1(X_{\ab},\CC)_{L}&0\\ }; 
\path[->,>=angle 90](m-1-1) edge node[auto] {} (m-1-2); 
\path[->,>=angle 90](m-1-2)edge node[auto,swap] {}(m-1-3); 
\path[->,>=angle 90](m-1-3)edge node[auto,swap] {}(m-1-4);
\end{tikzpicture} 
\end{equation} 
Since $\Spec \CC[K]$ has a canonical identification with
$\Char(K)$ and the dimension of the cokernel of the left
homomorphism in~\eqref{quotientresolution} is the sum of cokernels of
localizations of~\eqref{quotientresolution}
at the maximal ideal of $\xi \in \Char(K) \subset \Char(X)$, 
the dimension of cokernel in~\eqref{quotientresolution} 
is equal to $\sum_{\xi \in \Char(K)} d(i_H(\xi))$.
To conclude the proof we will show that contribution of the character
$\xi=1$ in the last sum is equal to the dimension of the image of the
left homomorphism in~\eqref{fiveterms} and that the right term
in~\eqref{fiveterms} is equal to $b_1(X)$. 
Indeed, since
$b_1(X)=\rk \pi_1(X)/\pi_1'(X)$,
$\dim H_1(L,\CC)=\rk L$ 
and the group
$K$ is finite the second claim follows. 
The first one follows from consideration of the commutative
square obtained by taking morphism of the sequence~\eqref{fiveterms} 
into similar five term sequence
replacing $H$ by $\pi_1(X)$: 
\begin{equation} 
\begin{tikzpicture}[description/.style={fill=white,inner sep=2pt},baseline=(current bounding box.center)] 
\matrix (m) [matrix of math nodes, row sep=3em, 
column sep=2.5em, text height=1.5ex, text depth=0.25ex] 
{ H_2(L,\CC)& H_1(X_{\ab})_L \\ 
H_2(\pi_1(X)/\pi_1'(X),\CC)&H_1(X_{\ab})_{\pi_1(X)/\pi_1'(X)}.\\}; 
\path[->,>=angle 90](m-1-1) edge node[auto] {} (m-1-2); 
\path[->,>=angle 90](m-1-1)edge node[auto,swap] {}(m-2-1); 
\path[->,>=angle 90](m-2-1)edge node[auto,swap] {}(m-2-2); 
\path[->,>=angle 90](m-1-2)edge node[auto,swap] {}(m-2-2); 
\end{tikzpicture} 
\end{equation} 
The left vertical arrow is isomorphism (again since $K$ is finite) 
and the right vertical
arrow is surjection which is the isomorphism
over the contribution of the trivial character in $H_2(L,\CC)$. 
Hence the identity~\eqref{formulabetti} is verified. 
\end{proof}

Definition~\ref{charvardef} allows algorithmic calculation 
of the characteristic varieties (using Fox calculus)
provided a presentation of the fundamental group is known. 
See for example~\cite{pisa,suciu-translated}
for explicit cases of such calculations.

On the other hand one has the following interpretation 
using local systems~(\cite{eko,charvar}). 
Recall that a (rank $n$) local system is a ($n$-dimensional)
linear representation of the fundamental group $\pi_1(X)$.
For treatment of the local systems and their cohomology 
we shall refer to~\cite{deligne}.

A topological definition of the cohomology of rank one 
local systems can be given as follows. 
If $X$ is a finite $CW$-complex, $\chi$ is a 
character of $\pi_1(X)$
and $X_{\ab}$ is the universal abelian cover,
then one can define the twisted cohomology $H^k(X,\chi)$ as the cohomology of the complex:
\begin{equation}
\ldots\longrightarrow C^k(X_{\ab}) \otimes_{\CC[H_1(X,\ZZ)]} \CC_{\chi}
\buildrel {\partial^k \otimes 1} \over 
\longrightarrow\ldots 
\end{equation}
where $C^k(X_{\ab})$ is the $\CC$-vector space of $i$-cochains of $X_{\ab}$ 
considered as
a module over the group ring  of $H_1(X,\ZZ)$ and 
the  $\CC_{\chi}$ is the one dimensional 
$\CC$-vector space with the 
$\CC[H_1(X,\ZZ)]$-module structure given by the character~$\chi$.
If $X$ is a smooth manifold, $H^*(X,\chi)$ has a de~Rham description (cf.~\cite{deligne}).
The homology of a local system can be described using the dual chain complex.
We have the following:

\begin{theorem}\label{otherdepthdef} 
If $\chi \ne 1$, then 
\begin{equation}
d(\chi)=\dim H^1(X,\chi). 
\end{equation}
\end{theorem}

The connection with the cohomology of local systems allows one to apply general
techniques on cohomology with twisted coefficients, which yield the following 
results on the structure of characteristic varieties:

\begin{theorem}[\cite{arapura}]\label{arapurath} 
Each $V_k(X)$ is a finite union of cosets of subgroups of $\Char(X)$.
Moreover, for each component $V$ of $V_k(X)$ having a positive dimension 
there is a map $f: X \rightarrow C$ where $C$ is a quasi-projective curve
such that $V$ is a coset of the subgroup $f^*H^1(C,\CC^*)\subset\Char(X)$. 
\end{theorem}

\begin{theorem}[\cite{acm-arapura}]\label{thm-structure}
Let $V$ be an irreducible component of $V_k(X)$. 
Then one of the two following statements holds:
\begin{enumerate}
\enet{\rm(\arabic{enumi})}
\item\label{thm-structure-orb} 
There exists an orbicurve $\cC$, a surjective orbifold morphism $f:X\to \cC$ 
and an irreducible component $W$ of $V_k^{\orb}(\cC)$ such that $V=f^*(W)$.
\item\label{thm-structure-tors} $V$ is an isolated torsion point not of type~\ref{thm-structure-orb}.
\end{enumerate}
\end{theorem}

Recall the definition of $V_k^{\orb}$ after Lemma~\ref{lemma-ofg} and see section~\ref{sec-orbicurves} 
for more details on $V_k^{\orb}(\cC)$ for orbicurves.

\subsection{Alexander polynomial associated with a character} \label{sec-alexpol} 
\mbox{} 
 
A specialization of the characteristic variety of a topological space $X$ of finite type to a
special character of its fundamental group $G:=\pi_1(X)$ can be defined and it is a natural generalization
of the Alexander polynomial to this context. For the sake of simplicity, as at the beginning of
section~\ref{sec-charvar}, we shall assume that an identification 
$G/G'\cong \ZZ^r$ was made. 
 
Let $X$ be a finite $CW$-complex and $\chi\in \Char(G)=(\CC^*)^r$
\footnote{this identification depends in a choice of generators of $H_1(X,\ZZ)$.}
a torsion character, that is $\chi:=(\xi_d^{\varepsilon_1},\dots,\xi_d^{\varepsilon_r})$, where $\xi_d$ is a 
primitive $d$-th root of unity, $0\leq \varepsilon_i< d$, and $d$ is the order of $\chi$. Note that $\chi$ 
determines naturally an epimorphism $\varepsilon:G/G'=\ZZ^r \to \ZZ$ defined as $\varepsilon(e_i):=\varepsilon_i$. 
Let $K_{\e}=\ker \e$, and $K_{\e}'=[K_{\e},K_{\e}]$ be the commutator of $K_{\e}$. By the Hurewicz Theorem 
$M_\chi:=K_{\e}/K_{\e}'$ can be identified with the homology of the infinite cyclic cover of $X$ corresponding 
to $\e$ and hence it can be viewed as a module over the group ring $\Lambda:=\QQ[t^{\pm 1}]$, where $t$ is a 
generator of the Galois group of covering transformations. 
 
\begin{dfn} 
Let $X$ and $\chi$ be as above, then the \emph{Alexander polynomial of $X$} associated with $\chi$ is a
generator of the order of the module $M_\chi$ and will be denoted by~$\Delta_{X,\chi} (t)$. 
\end{dfn} 
 
The following is a direct consequence of the definition and Theorem~\ref{otherdepthdef}. 
 
\begin{prop} 
\label{prop-alex-roots} 
Under the above conditions, if $\chi\neq 1$, then $d(\chi)$ is the multiplicity of the factor $\phi_d(t)$ in
$\Delta_{X,\chi}(t)$, where $\phi_d(t)$ is the cyclotomic polynomial of order~$d$. 
\end{prop} 
 
\begin{proof} 
Using the same arguments as in~\cite[Theorem~2.26]{kike-hiro-survey}, the polynomial $\Delta_{X,\chi}$
is the order of the torsion of
$$ 
(G'/G'' \otimes \QQ) \otimes_{\Lambda}\Lambda/ {(t_1-t^{\e(\gamma_1)},\dots,t_r-t^{\e(\gamma_r)})}, 
$$ 
therefore $d(\chi)$ is the multiplicity of $\xi_d$ as a root of $\Delta_{X,\chi}$. Since
$\Delta_{X,\chi}\in \QQ[t]$, the result follows. 
\end{proof}

\subsection{Weight of a character} 
\mbox{} 
 
Now let us assume that $X$ is a smooth quasi-projective variety 
and let $\chi \in \Char(X)$ be a character of finite order. 
Let $X_{\chi}$ be the covering space corresponding to $\ker \chi \subset
\pi_1(X)$. Then $H^1(X_{\chi})$ supports a mixed Hodge structure
with weights $1,2$ (cf.~\cite{delignehodge}). The cyclic
group $\im(\chi)$ acts (freely) on $X_{\chi}$ preserving
both the Hodge and weight filtration. 
\begin{dfn}\label{weightcharacter}
An integer 
$w$ is called a weight of a character $\chi$ if the $\chi$-eigenspace
of a generator $g$ of $\im(\chi)$ acting on $\Gr^W_wH^1(X_{\chi})$ has
a positive dimension. Similarly, $p$ is called a Hodge filtration of $\chi$
if the $\chi$-eigenspace of automorphism of $\Gr^p_FH^1(X_{\chi})$ induced by $g$
has a positive dimension. 
\end{dfn} 

The following gives an expression for the weight of a character 
in terms of finite coverings with arbitrary Galois groups. 
 
\begin{prop}\label{weightonanycover} 
 Let $X_G \rightarrow X$ be an abelian cover of $X$
with a finite covering group $G$. Let $\chi \in \Char(G)$.
Then
\begin{equation}
\dim (W_wH^1(X_G))_{\chi}=\dim W_wH^1(X_{\chi}). 
\end{equation} 
In particular $G$ has non-zero $\chi$-eigenspace on $W_wH^1(X_G)$ 
iff $\chi$ has weight~$w$. 
\end{prop} 
 
\begin{proof} The argument is similar to the one in the proof of 
Theorem~\ref{bettinumberofcover}. 
\end{proof}
 
\begin{remark} 
Note that characters might have either no weights or more than one weight. 
In this paper we will be most interested in characters with only one weight, 
namely weight~2. 
\end{remark}

\subsection{Essential, non-essential and essential coordinate 
components}\label{nonessential}
\mbox{} 

The constructions described in the previous section can be applied to the case when 
$X=\PP^2\setminus \bigcup_{i=0}^{r} D_i$, where $D_i$ are irreducible curves. 
In this case $H_1(X,\ZZ)=\ZZ^{r+1}/(d_1,...,d_r)$, where $d_i=\deg D_i$. If the degree
of one of its components, say $D_0$, is equal to one, 
i.e. we have the complement to a plane curve in $\CC^2$, then $H_1(X,\ZZ)$ 
is a free abelian group of rank $r$. Let $D=\bigcup D_i$ denote the (reducible) 
curve in $\CC^2$ formed by irreducible components $D_i$.
One has a preferred surjection
$\pi_1(\CC^2\setminus D) \rightarrow \ZZ^r$ given by the 
linking numbers of a loop representing the element of $\pi_1$ with the 
component~$D_i$:
$$ 
\gamma \mapsto (\ldots,\lk(\gamma,D_i),...\ldots). 
$$
This also yields the identification $\Char(\CC^2\setminus D)=(\CC^*)^r$.

Let $D'$ be a reducible subcurve of $D$ in $\CC^2$, that is $D'\subset D$.
Then (cf.~\cite{charvar}) one has a surjection
$\pi_1(\CC^2\setminus D) \rightarrow \pi_1(\CC^2\setminus D') \rightarrow 1$ and
hence an embedding 
\begin{equation}
i_{D'}: \Char(\CC^2\setminus D') \rightarrow \Char(\CC^2\setminus D).
\end{equation}
The image of $i_{D'}$ is formed by the factors of $(\CC^*)^r$ corresponding 
to the components of $D'$. 
Moreover it was shown in~\cite{charvar}
that if $\chi \in V_k(\CC^2\setminus D')$ then $i_{D'}(\chi) \in V_j$
with~$j \ge k$.

\begin{dfn}[\cite{charvar}]\label{essentialdef} 
The components of the characteristic variety $V_k(D)$ 
obtained as the image of a component of $V_k(D')$ are called \emph{non-essential}.
A component of $V_k$ is called \emph{coordinate} if it belongs to 
$i_{D'}(\Char(\CC^2\setminus D'))$ for some $D' \subsetneq D$.
\end{dfn}

Given a surjection
$\pi_1(\CC^2\setminus \bigcup C_k) \rightarrow G:=\ZZ_{a_1} \oplus\dots\oplus \ZZ_{a_m}$
(for $m \le r$), one can construct the unbranched covering 
space $X_{a_1,...a_m}$ corresponding to the kernel of the above surjection of the 
fundamental group. Moreover, there is a compactification of this unbranched cover and its 
morphism to $\PP^2$ extending the covering map. Though this compactification (branched
cover) is non-unique, its birational class is well defined. In particular 
the first Betti number of this branched cover is well defined. 
A compactification $\bar X_{a_1,...a_m}$
can be selected so that it supports a $G$-action 
extending the action of $G$ on $X_{a_1,...a_m}$. A calculation of 
$H_1(\bar X_{a_1,...a_m}, \CC)$ as a $G$-module is given by the following:

\begin{theorem}[\cite{sakuma}] \label{sakumastatement} 
 For each character $\chi$ of $G=\ZZ_{a_1} \oplus 
...\oplus \ZZ_{a_m}$ let 
\begin{equation}\label{eigenspace1}
W_{\chi}:=\{v \in H_1(\bar X_{a_1,...a_m},\CC) \mid g \cdot v= \chi(g)v 
\ for \ any \ g \in G\}
\end{equation}
be the eigenspace of $G$-action corresponding to the character $\chi$.
Let $D_{\chi}$ be the union of components of $D$ over which 
the character $\chi$ is unramified \footnote{a character 
$\chi \in H^1(\CC^2\setminus D)$ is unramified along a component $D_i \subset D$ 
if for the boundary $\gamma_i \in H_1(\CC^2\setminus D)$ 
of a small disk transversal to $D_i$ one has $\chi(\gamma_i) = 1$.
The characters of $\pi_1(\CC^2\setminus D)$ unramified along $D_i$ can be identified 
with the characters of $\pi_1((\CC^2\setminus D) \cup D_i)$.}.
Then $\dim W_\chi$ is equal to the depth of $\chi$ considered 
as the character of $\pi_1(\CC^2\setminus D_{\chi})$. 
\end{theorem}

This theorem was used in~\cite{charvar} to 
describe essential components of $\pi_1(\CC^2\setminus D)$ in terms of 
combinatorics of singularities of $D$ and the superabundances 
of the linear systems of curves given by the local type of
singularities, their position on $\PP^2$ and the degree of~$D$. 
 
\section{Albanese varieties of smooth quasi-projective varieties.}
 
Recall (cf.~\cite{serre})
that given a projective variety $X$ there is a canonically
associated abelian variety $\Alb(X)$ and the map $X \rightarrow \Alb(X)$ 
(unique up to a choice of the image of point in $X$) is universal
with respect to the maps into abelian varieties, i.e.
 given an abelian variety $A$ and a morphism
$X \rightarrow A$ there is a unique (up to ambiguity as above)
factorization $X \rightarrow \Alb(X) \rightarrow A$. 
 
This construction
can be extended to the quasi-projective case so that the Albanese variety
is a semiabelian variety which
is universal with respect to the morphisms into algebraic groups. 
For example one can use Deligne's construction of 1-motif
associated to the mixed Hodge structure on cohomology $H^1(X)$
of smooth quasi-projective variety (cf.~\cite{deligne}). 
More precisely one has the following: 

\begin{theorem}\label{albanese}
Let $X$ be a quasi-projective variety which 
is a complement to a divisor with normal crossings in 
a smooth projective variety $\bar X$.

\begin{enumerate} 
\enet{\rm(\arabic{enumi})} 
 \item Then one has exact sequence:
\begin{equation}\label{albanesesequence}
0 \rightarrow \cA \rightarrow \Alb(X) \rightarrow \Alb(\bar X) 
\rightarrow 0
\end{equation}
where $\cA$ is an affine abelian algebraic group 
isomorphic to a product of $\GG_m$.
$Alb(X)$ depends on $X$ functorially i.e. a
 morphism $X_1 \rightarrow X_2$ induces the 
homomorphism $\Alb(X_1) \rightarrow \Alb(X_2)$.

\item
If $\Gamma$ is a finite group of biholomorphic automorphisms
then the sequence \eqref{albanesesequence} 
is compatible with the action of~$\Gamma$.
\end{enumerate}
\end{theorem}
 
An explicit construction can be given as follows
(cf. also~\cite{iitaka-logarithmic}). 
Let $\bar X$, as above, be a smooth compactification of $X$ 
such that $\bar X\setminus X=D=\bigcup D_i$ is a union
of smooth divisors having normal crossings. Then: 
\begin{equation} 
\Alb(X)=H^0(\bar X,\Omega^1_{\bar X}(\log D))^*/H_1(X,\ZZ),
\end{equation} 
where the embedding of $H_1(X,\ZZ)$ as a lattice is given by $\gamma (\omega):= 
\int_{\gamma}(\omega)$. The Albanese map is given by 
\begin{equation}\label{albformula} 
 P \mapsto \int_{P_0}^P\omega
\end{equation} 
(here $P_0$ is a fixed point on $X$).
The integral~\eqref{albformula} depends (modulo periods of $\omega$) 
only on the end points of the path since a holomorphic
logarithmic form is closed (cf.~\cite{deligne}).
One has the commutative diagram: 
\begin{equation}\label{albanesediagram}
\begin{tikzpicture}[description/.style={fill=white,inner sep=2pt},baseline=(current bounding box.center)] 
\matrix (m) [matrix of math nodes, row sep=3em, 
column sep=1.2em, text height=1.5ex, text depth=0.25ex] 
{0& \ker\left(\bigoplus_i H^0(\cO_{D_i})\rightarrow H^1(\Omega^1_{\bar X})\right)^* &
H^0(\Omega^1_{\bar X}(\log D))^* &
H^0(\Omega^1_{\bar X})^*& 0 \\ 
0 & \coker\left(H_2(\bar X,\ZZ) \rightarrow \bigoplus_i H_0(D_i,\ZZ)\right)&
H_1(X,\ZZ) & H_1(\bar X,\ZZ)& 0\\}; 
\path[->,>=angle 90](m-1-1) edge node[auto] {} (m-1-2); 
\path[->,>=angle 90](m-1-2) edge node[auto] {} (m-1-3); 
\path[->,>=angle 90](m-1-3) edge node[auto] {} (m-1-4); 
\path[->,>=angle 90](m-1-4) edge node[auto] {} (m-1-5); 
\path[->,>=angle 90](m-2-1) edge node[auto] {} (m-2-2); 
\path[->,>=angle 90](m-2-2) edge node[auto] {} (m-2-3); 
\path[->,>=angle 90](m-2-3) edge node[auto] {} (m-2-4); 
\path[->,>=angle 90](m-2-4) edge node[auto] {} (m-2-5); 
\path[->,>=angle 90](m-2-2)edge node[auto,swap] {}(m-1-2); 
\path[->,>=angle 90](m-2-3)edge node[auto,swap] {}(m-1-3); 
\path[->,>=angle 90](m-2-4)edge node[auto,swap] {}(m-1-4); 
\end{tikzpicture} 
\end{equation} 
In this diagram the upper row is dual to
the exact cohomology sequence corresponding
to the sequence of sheaves given by the residue map
(cf.~\cite[(3.1.5.2)]{deligne};
below $j: D_i \rightarrow \bar X$): 
\begin{equation} 
 0 \rightarrow \Omega^1_{\bar X} \rightarrow \Omega^1_{\bar X}(\log D) 
\stackrel{\bigoplus_i \res_{D_i}}{\longrightarrow}
\bigoplus_i j_*\cO_{D_i} \rightarrow 0.
\end{equation} 
The lower row is the exact sequence of the pair $(\bar X,X)$ in which 
we used the identification:
\begin{equation} 
H_2(\bar X,X,\ZZ)=H^{2 \dim D}(D,\ZZ)= 
\bigoplus_i H^{2 \dim D}(D_i,\ZZ)=\bigoplus_i H_0(D_i,\ZZ).
\end{equation}
One verifies that all vertical arrows are injective and the image of
each provides the lattice in the corresponding complex vector space 
in the upper row. The rank of the lattice which is the image of
the right (resp. left) 
vertical row is equal to the real (resp. complex) dimension of the
target\footnote{Note that the fact that one uses the \emph{real} 
dimension of $H^0(\Omega^1_{\bar X})$ is that
$\rk H_1(X,\ZZ)=\dim_{\CC}H^1(\cO_{\bar X})+\dim H^0(\Omega^1_{\bar X}(\log D))$ 
which follows from the degeneration of the Hodge-deRham spectral sequence 
in the $E_1$-term.}. 
In the case of the left arrow, one uses that the dual map to 
$H_2(\bar X,\CC)
\rightarrow \bigoplus_i H_0(D_i,\CC)$
factors as
$$ 
\bigoplus_i H^0(D_i,\CC) \rightarrow H^{1,1}(\bar X) \rightarrow
H^{2}(\bar X,\CC). 
$$ 
Hence the quotient of the right (resp. left) injection is
a complex torus (resp. affine algebraic group isomorphic to a product of
several copies of $\CC^*$). This complex torus is isomorphic
to the Albanese variety of $\bar X$ by the classical construction. 
The remaining assertions of the Theorem~\ref{albanese} 
follow from the description of the Albanese map given by~\eqref{albformula}. 
 
Finally note that Theorem~\ref{albanese} implies the following:
 
\begin{cor}\label{involutiondecomp} 
 Let $\phi$ be an involution of $X$, $\phi_*$ be the
corresponding automorphism of $\Alb(X)$. 
Let $\Alb(X)^-=\{v \in \Alb(X) \mid i \phi_*(v)=-v\}$ and 
$\Alb(X)^+=\{v \in \Alb(X) \mid i \phi_*(v)=v\}$. 
Then one has an isogeny: 
\begin{equation} 
\Alb(X)=\Alb(X)^-\oplus \Alb(X)^+ 
\end{equation} 
\end{cor} 
 
\section{Orbifold pencils and characters of fundamental groups 
of quasi-projective manifolds} 

\subsection{Orbicurves} \label{sec-orbicurves} 
 
\begin{dfn} 
\label{def-global}
An \emph{orbicurve} is a complex orbifold of dimension equal to one 
\footnote{i.e. a smooth complex curve with a collection $R$ of points 
(called {\it the orbifold points}) 
with a multiplicity assigned to each point in $R$. The complement 
to $R$ is called {\it the regular part} of the orbifold.}.
An orbicurve $\cC$ is called a \emph{global quotient} if there exists a finite group $G$
and a manifold $C$ such that $\cC$ is the quotient of $C$ by $G$
with standard orbifold structure.

A \emph{marking} on an orbicurve $\cC$ (resp. a quasi-projective variety $\cal X$) is a character of 
its orbifold fundamental group
\footnote{Recall (cf.~\cite{adem})
that the orbifold fundamental group $\pi_1^{\orb}(\cC)$ 
of an orbifold $\cC$ is defined as the
quotient $\pi_1(\cC\setminus \cR)$ by the normal closure of the elements
$\gamma_{p_i}^{m(p_i)}$ where $m(p_i)$ is the
multiplicity of an orbifold point $p_i$ and $\gamma_i$ is a meridian of $p_i$.} 
(resp. its fundamental group) that is, an element of 
$\Char^{\orb}(\cC):=\Hom(\pi_1^\orb(\cC),\CC^*)$ (resp. $\Char(X)$)
(warning: this terminology is different from the one used in~\cite{adem}).

A \emph{marked orbicurve} is a pair $(\cC,\rho)$, where $\cC$ is an orbicurve and $\rho$
is a marking on~$\cC$. More generally, 
one defines a \emph{marked quasi-projective manifold} as a pair $(X,\chi)$ 
consisting of a quasi-projective manifold $X$ and a character of its 
fundamental group. From now on, all characters used as marking below will be assumed to 
have a finite order. 

A marked \emph{global quotient} is a marked orbicurve $(\cC,\rho)$
such that if $C\setminus R \rightarrow \cC\setminus \cR$ is the unbranched cover 
corresponding to the global quotient $C \rightarrow \cC$,
with $R$ being the set of fixed points of non-identity elements of the
covering group and $\cR$ being its image (or equivalently the
set of orbifold points), then one has 
\begin{equation} 
\pi_1(C\setminus R)=\ker (\pi_1(\cC\setminus \cR)\to \pi_1^\orb (\cC)
\rightmap{\rho} \CC^*) 
\end{equation} 
In other words, the above cover
$C\setminus R \rightarrow \cC\setminus \cR$ 
is the cover of the
minimal degree over which $\rho$ becomes trivial. 
\end{dfn}
 
\begin{remark} The above existence condition for a marking on a global quotient
implies that the quotient map over the regular part of the orbifold is a
{\it cyclic} cover. 
More precisely, the covering group of the cover of the regular part of
the orbifold via $\rho$ can be identified with $\im(\rho) \subset
\CC^*$ i.e. $\rho$ can be viewed as a character of the covering group. 
If $d$ is the order of this covering group then the
number of possible markings is equal to the value of the Euler function
$\phi(d)$.
\end{remark}

\begin{dfn} Let $\cC$ be a global orbifold quotient and $\rho$ 
a marking. Let $R$ be the set of orbifold points and $C\setminus R \rightarrow
\cC\setminus \cR$ be the quotient map with the covering group $G$. 
The integer 
\begin{equation} 
 d(\rho)=\dim \{v \in H^1(C\setminus R,\CC) \mid g\cdot v=\rho(g)v, \ \ g \in G \} 
\end{equation} 
is called the \emph{depth} of a character $\rho$ of the orbicurve $\cC$.
\end{dfn} 
 
\begin{example}\label{ex-orbifold} Let $\CC_{n,n}$ be the orbifold supported on
$\CC$ with two orbifold points of multiplicity $n$. 
We shall identify $\CC$ with $\PP^1\setminus \{[1:1]\}$ so that the orbifold
points correspond to $[0:1],[1:0]$. 
This is the global quotient of a smooth curve $C$ by the cyclic group 
$\ZZ/n$ where $C$ is the complement in
$\PP^1$ to the set $S:=\{[\xi_n^i:1]\mid i=0,1\dots,n-1\}$ of $n$ points 
(here $\xi_n$ is a primitive root of unity of degree $n$) and
the global quotient map is the restriction on the complement to $S$ 
of the map $\PP^1 \rightarrow \PP^1$ 
given by $z \mapsto z^n$.
We have $\pi_1^{\orb}(\CC_{n,n})=\ZZ/n*\ZZ/n$. Such free product decomposition
implies that $V_1^{\orb}(\CC_{n,n})= 
\mu_n \times \mu_n$, where $\mu_n$ is the multiplicative cyclic group of order $n$. 
Consider a character $\rho \in V_1^{\orb}(\CC_{n,n})$
taking values $\zeta,\zeta^{-1}$ on respective generators of this direct sum 
where $\zeta$ is a primitive root of unity.
It follows that if $\pi_1(\PP^1\setminus \{[1:0],[1:1],[0:1]\}) \rightarrow \pi_1^{\orb}(\CC_{n,n})$ 
is the canonical surjection (in the above
identification of $\CC$ and $\PP^1$ so that
the point at infinity corresponds to $[1:1]$),
then the pullback of $\rho$ takes values $\zeta, 1,\zeta^{-1}$
on generators corresponding to $[1:0],[1:1],[0:1]$.
In particular the covering space corresponding to such $\rho$ is
$\PP^1\setminus S$ and the dimension of the $\rho$-eigenspace is equal to one. 
\end{example}

\subsection{Orbifold pencils}\label{quasiprojpencils}

\begin{dfn} 
\label{def-orb-pencil}
Let $\cX$ be a quasi-projective manifold and $\cC$ be an orbicurve. 
A holomorphic map $\phi$ between $\cX$ and the underlying $\cC$ complex curve
is called an \emph{orbifold pencil} if the index of each orbifold point $p$ 
divides the multiplicity of each connected component of the fiber $\phi^*(p)$
over~$p$.
\end{dfn}

\begin{remark}
Note that this definition implies that if $\Gamma_i$ is the boundary of
a small disk normal to $\phi^{-1}(p_i)$ at its smooth point
then $\phi(\Gamma_i)$ belongs to
the subgroup of $\pi_1(\cC\setminus p_i)$ generated by $\gamma_i^{m(p_i)}$.
In particular an orbifold pencil induces the map
$\pi_1(\cX) \rightarrow \pi_1^{\orb}(\cC)$. 
\end{remark}

\begin{dfn} 
\label{def-orb-marking}
Let $\cal X$ be a quasi-projective variety, $C$ be a quasi-projective
curve and $\cC$ be an orbicurve
which is a global quotient of $C$. A \emph{global quotient orbifold pencil} 
is an orbifold pencil
$\phi: \cal X \rightarrow \cC$ such that there exists a morphism $\Phi: X_G \rightarrow C$,
where $X_G$ is a quasi-projective manifold endowed with an action of the group $G$
which makes the diagram:
\begin{equation}\label{diagramorbchar}
\begin{tikzpicture}[description/.style={fill=white,inner sep=2pt},baseline=(current bounding box.center)] 
\matrix (m) [matrix of math nodes, row sep=2.5em, 
column sep=2.5em, text height=1.5ex, text depth=0.25ex] 
{X_G & C \\ 
\mathcal{X} & \cC\\}; 
\path[->,>=angle 90](m-1-1) edge node[auto] {$\Phi$} (m-1-2); 
\path[->,>=angle 90](m-2-1) edge node[auto] {$\phi$} (m-2-2); 
\path[->,>=angle 90](m-1-1)edge node[auto,swap] {}(m-2-1); 
\path[->,>=angle 90](m-1-2)edge node[auto,swap] {}(m-2-2); 
\end{tikzpicture} 
\end{equation} 
commutative, for which the vertical arrows are the quotients by the action of $G$.

If, in addition, $(\cX,\chi)$ and $(\cC,\rho)$ are marked, then the global quotient orbifold pencil 
$\phi: \cal X \rightarrow \cC$ is \emph{marked} if $\chi=\phi^*(\rho)$ 
where $\phi^*: \Char^{\orb}(\cC) \rightarrow \Char(\cX)$
is the homomorphism dual to the surjection $\phi_*: \pi_1(\cX) \rightarrow
\pi_1^{\orb}(\cC)$ corresponding to the orbifold map $\phi$.
We will refer to the map of pairs $\phi:(\cX,\chi)\to (\cC,\rho)$ as a 
\emph{marked global quotient orbifold pencil} in $(\cX,\chi)$ of target $(\cC,\rho)$.
\end{dfn}
 
\begin{remark}\label{pullbackremark} 
Consider $R$ the collection of non-manifold points in $\cC$ and $F=\phi^{-1}(R)$ the collection
of multiple fibers corresponding to $\phi$. The orbifold relation $\chi=\phi^*(\rho)$ takes place
if and only if the following holds: 
\begin{equation} 
\hat \chi =i^*(\chi) \ \ \ \hat \chi=\hat \phi^*(\hat \rho) 
\ \ \hat \rho=p^*(\rho), 
\end{equation} 
for the map of open manifolds $\hat \phi: \cX\setminus F \rightarrow \cC\setminus R$,
induced by $\phi$, where $i: \cX\setminus \phi^{-1}(R) \rightarrow \cX$ is the embedding
and $p: \pi_1(\cC\setminus R) \rightarrow \pi_1^{\orb}(\cC)$ is the canonical projection.
\end{remark} 
 
\begin{lemma}\label{extensionexistence} 
Let $\cX$ be a quasi-projective manifold, $\cC$ a 
marked global quotient with marking $\rho \in \Char^{\orb}(\cC)$. 
If $\phi: \cX \rightarrow \cC$ is an orbifold map such that 
$\chi=\phi^*(\rho)$ then $\phi$ is a global quotient orbifold pencil 
i.e. $\phi$ can be extended to a commutative diagram~\eqref{diagramorbchar}. 
\end{lemma} 
 
\begin{proof} Let, as above, $\pi_G: X_G \rightarrow \cX$ be the covering space
corresponding to $\ker \chi \subset \pi_1(\cX)$.
Since, as follows from Remark~\ref{pullbackremark},
$\hat {\chi}=\hat \phi^*(\hat \rho)$ 
 one has the commutative diagram: 
 \begin{equation} 
\begin{tikzpicture}[description/.style={fill=white,inner sep=2pt},baseline=(current bounding box.center)] 
\matrix (m) [matrix of math nodes, row sep=3em, 
column sep=2.5em, text height=1.5ex, text depth=0.25ex] 
{ X_G\setminus \pi_G^{-1}(F)& & C\setminus R \\ 
\mathcal{X}\setminus F&& \cC\setminus \cR \\ }; 
\path[->,>=angle 90](m-1-1) edge node[auto] {$\Phi $} (m-1-3); 
\path[->,>=angle 90](m-1-1)edge node[auto,swap] {$i^*(\pi_G) $}(m-2-1); 
\path[->,>=angle 90](m-1-3)edge node[auto] {}(m-2-3); 
\path[->,>=angle 90](m-2-1)edge node[auto] {$\phi$}(m-2-3); 
\end{tikzpicture} 
\end{equation} 
where $i^*(\pi_G)$ is the restriction of $\pi_G$ onto $X_G\setminus \pi_G^{-1}(F)$. 
We have to verify that $\Phi$ extends to the map $X_G \rightarrow C$ 
i.e. that is any pair of small loops
$\tilde \Gamma_P,\tilde \Gamma_P' \in \pi_1(X_G\setminus \pi_G^{-1}(F))$
about
points $P,P'$ in a connected component $F'$ 
of $\pi_G^{-1}(F)$ is mapped by $\Phi$ into a pair small loops
$\tilde \delta, \tilde \delta'$,
about the same point in $R\subset C$.
The images of loops $\tilde \delta, \tilde \delta'$ in
$\cC\setminus \cR$ are homotopic and hence they will be homotopic in
$C\setminus R$ iff $\tilde \delta^{-1} \cdot \tilde \delta'$ belongs 
to the subgroup of $\pi_1(\cC\setminus \cR)$ generated by $\delta^m$ 
(since preimages of points of $\cR$ in $C$ correspond to
cosets in the covering group of $C\setminus R$ over $\cC\setminus \cR$ of the
subgroup generated by the image of $\delta^m$). 
So we claim that $\phi_* \circ (i^*(\pi_G))_*(\tilde \Gamma_P), 
\phi_* \circ (i^*(\pi_G))_*(\tilde \Gamma_{P'})$ belong to the
same coset of $(\delta^m)$. 
Indeed, since $P,P'$ belong to a connected curve $F'$, 
$(i^*(\pi_G))_*(\tilde \Gamma_P^{-1}\tilde \Gamma_P') \in \pi_1(\partial F)$ 
where $\partial F$ is the boundary of a small neighborhood of
a component of $F$. Since one has an exact sequence 
$(\Gamma) \rightarrow \pi_1(\partial F) \rightarrow \pi_1(F) \rightarrow 0$, 
where $\Gamma$ is a small loop about the component of $F$,
and the image of a lift into $\partial F$ of 
a loop in a component of $F$ is trivial in $\cC$, 
the image of $\pi_1(\partial F)$ in $\pi_1(\cC)$
belongs to the subgroup generated by $\phi_*(\Gamma)=\delta^m$ and hence
images of $\tilde \Gamma_P, \tilde \Gamma_{P'}$ are in the same coset
of $\delta^m$ as was claimed. 
\end{proof} 
 
\begin{remark} The key point in the above argument is that 
images of $\Gamma_P,\Gamma_{P'}$ are the same in $\pi_1^{\orb}(\cC)$ 
and that the covering group of $C\setminus R \rightarrow \cC\setminus \cR$ is the
quotient of the latter. 
\end{remark} 
 
\begin{dfn}
\label{def-indep}
Global quotient orbifold pencils $\phi_i:(\cal X,\chi) \rightarrow (\cC,\rho)$, $i=1,...,n$
are called \emph{independent} if the induced maps $\Phi_i: X_G \rightarrow C$ 
constructed in Lemma~\ref{extensionexistence} 
define $\ZZ[G]$-independent morphisms of modules
\begin{equation}
{\Phi_i}_*: H_1(X_G,\ZZ) \rightarrow H_1(C,\ZZ).
\end{equation}

In addition, if $\bigoplus_i {\Phi_i}_*: H_1(X_G,\ZZ) \rightarrow H_1(C,\ZZ)^n$ is surjective we say that 
the pencils $\phi_i$ are \emph{strongly independent}.
\end{dfn}

\begin{remark}
Note that if either $n=1$ or $H_1(C,\ZZ)=\ZZ[G]$, then independence is
equivalent to surjectivity of $\bigoplus_i {\Phi_i}_* \otimes \QQ$
since the matrix of the latter has as its columns the vectors
corresponding to ${\Phi_i}_*$. 
\end{remark}

This definition is motivated by the following:

\begin{prop} 
\label{prop-indep}
Let $\cC$ be an orbicurve which is a global $G$-quotient of 
the algebraic group $\cA=\CC^*$. The global quotient orbifold pencils 
$\phi_i: \cal X \rightarrow \cC$, $i=1,...,r$ on a global quotient
orbifold $\cX=X_G/G$ such that the first Betti number of
a smooth compactification of $X_G$ is zero,
are independent in the sense of Definition{\rm~\ref{def-indep}} if and only
if they define $\ZZ$-independent 
elements of the abelian group $\Mor_G(X_G,\cA)$ of equivariant morphisms. 
\end{prop}

\begin{proof}
Note that
$\Mor_G(X_G,\CC^*)=\Hom_G(\Alb(X_G),\CC^*)$ by the
universal property of maps from Albanese 
into an algebraic group. The assumption on the first Betti number of
compactification of $X_G$ yields that $\Alb(X_G)$ is
a torus of dimension $b_1(X_G)$.
Also $\Hom_G(\Alb(X_G),\CC^*)=\Hom(H_1(X_G;\ZZ),H_1(\CC^*;\ZZ))$
holds equivariantly and the claim follows.
\end{proof}

\begin{remark} In \cite{mordweil} it was shown that the
Proposition~\ref{prop-indep} is also true for special $\cal X$ 
in cases when
 $\cA$ is a certain elliptic curve (depending on the Alexander polynomial 
of $\cal X$). 
This is so if $\cal X$ is the complement of a cuspidal curve $C$ in $\PP^2$. 
In this case 
$\Mor_G(X_G,\cA)$ can be identified with the Mordell-Weil group 
of $K$-points 
of the elliptic curve over $\CC$ admitting an automorphism of 
order 6 where $K$ is 
the field of rational functions on the 6-fold cover of $\PP^2$ ramified along
the curve i.e. the degree six extension of $\CC(x,y)$.
 This is also the case for $\delta$-curves discussed in ~\cite{mordweil}.
\end{remark}
 
On the other hand, in the case of orbifolds with a trivial orbifold structure
there are very few marked orbifold pencils. 
 
\begin{lemma} Let $\phi_i: (\cX,\chi) \rightarrow (C,\rho)$ be a collection
of strongly independent marked 
pencils. Assume that $C$ has the trivial orbifold structure. 
Then $\rho=1$ and hence $\chi=1$. 
\end{lemma} 
 
\begin{proof} Consider the map $(...,\phi_i,...): \cX \rightarrow C^n$ induced by pencils
$\phi$. Independence implies that the induced map $H_1(\cX) \rightarrow H_1(C^n)$
is surjective and hence the dual map of cohomology is injective.
If $p_i: C^n \rightarrow C$ is the projection on the $i$-th factor, then
$p_i^*(\rho)=(...,\rho,...)$ (the non-identity component is on the $i$-th coordinate). 
The compatibility condition together with the injectivity of $H^1(C,\CC^*) \rightarrow H^1(\cX)$ 
implies that $p_i^*(\rho)=\pi_j^*(\rho)$ and hence~$\rho=1$. 
\end{proof} 

\begin{lemma} 
Let $\phi$ be a marked orbifold pencil of $(\cal X,\chi)$ with target $(\cC,\rho)$ having connected fibers. 
If $\cC$ is a global quotient of a curve $C$ then $\phi$ is a marked global quotient orbifold pencil.
\end{lemma}

\begin{proof}
Let, as above, $\bar{\cal X}$ denotes a smooth compactification of $\cal X$.
Let $R \subset \cC$ be the set of non-manifold points.
 The map $\pi_1(\cal X\setminus \phi^{-1}(R)) \rightarrow \pi_1(\cC\setminus R)$ 
is surjective since the fibers of $\phi$ are connected. 
Let $K$ be the kernel of the homomorphism 
$\pi_1(\cal X\setminus \phi^{-1}(R)) \rightarrow \pi_1(\cal X) \rightarrow \mu_n$,
where the right homomorphism is the character $\chi$.
Consider $\tilde{X}$ the unbranched cover of $\cal X\setminus \phi^{-1}(R)$ 
corresponding to $K$. Extend it to a branched cover $\pi:\bar X \to \bar{\cal X}$ and let
$X:=\pi^{-1} \cal X$. Since $\chi$ is the image of 
$\rho \in \Char^{\orb}(\cC)$, the map 
$\cal X\setminus \phi^{-1}(R) \rightarrow \cC\setminus R$ extends to 
the map of unbranched covering spaces and thus to the map of branched covering spaces
$X \rightarrow C$ presenting the pencil $\phi$ as the global quotient.
\end{proof}

\subsection{Orbifold pencils, depth, and roots of Alexander polynomials} 
\label{quasiprojpencils-sec-thm} 
\mbox{} 
 
In this section we shall give a proof of Theorem~\ref{thm-main}.
 
%
 
\begin{proof}[Proof of Theorem{\rm~\ref{thm-main}}]
In order to prove part~\ref{thm-main-part1}, let us consider $\phi_{1,*},\dots,\phi_{n,*}$, 
$n$ strongly independent orbifold pencils. Since
$H_1(X_G;\ZZ) \longrightmap{\bigoplus_i \Phi_{i,*}} H_1(C;\ZZ)^n$ is an equivariant
epimorphism, the dual morphism $H^1(C;\ZZ)_\rho^n \to H_1(X_G;\ZZ)_\chi$ is injective
and hence $\rk H^1(C;\ZZ)_\rho^n=nd(\rho)\leq \rk H_1(X_G;\ZZ)_\chi = d(\chi)$.
 
As for part~\ref{thm-main-part2}, let $\chi$ be a 2-torsion character.
We shall apply the formula for the first
Betti number in Theorem~\ref{bettinumberofcover} to the degree-two covering
$X_{\chi}$ of $\cX$ corresponding to the subgroup
$\ker \chi$ of $\pi_1(\cX)$.
The sum in~\eqref{formulabetti} contains
only one term and yields that $d(\chi)$ is the dimension 
of the $\chi$ eigenspace of $H_1(X_{\chi},\CC)$: 
\begin{equation} 
d(\chi)= \{ v \in H_1(X_{\chi},\CC) \mid g\cdot v=\chi(g) v\}. 
\end{equation} 
Moreover, the action of $g \in \mu_2, g \ne 1$ is the multiplication 
by $-1$. The action of $g$ on $\Alb(X_{\chi})$ induces
the isogeny $\Alb(X_{\chi})=\Alb(X) \oplus \Alb(X_{\chi})^-$ 
where the second summand is the subvariety of $\Alb(X_{\chi})$ of points on
which the covering group acts as multiplication by $-1$
(cf. Corollary~\ref{involutiondecomp}; note that $\Alb(X_{\chi})^+= \Alb(X)$).
Since we assume that $\chi$ has only weight two, $\Alb(X_{\chi})^-$ has no compact part. 
In particular, by Theorem~\ref{albanese}, 
\begin{equation} 
\label{eq-alb} 
\Alb(X_{\chi})^-=\ker(\Alb(X_{\chi}) \rightarrow \Alb(X))=(\CC^*)^{d(\chi)} 
\end{equation} 
and the order two action is given by $z\mapsto z^{-1}$. Therefore the projections give $d(\chi)$ 
independent equivariant (due to Albanese functoriality) maps $X_{\chi} \rightarrow \CC^*$
Since $\Mor_{\ZZ_2}(X_{\chi},\CC^*)=\Hom_{\ZZ_2}(\Alb(X_{\chi}),\CC^*)$ and 
$\Hom_{\ZZ_2}({(\CC^*)}^{d(\chi)},\CC^*)=\Hom(\ZZ^{d(\chi)},\ZZ)$, then
$\rk_{\ZZ} \Mor_{\ZZ_2}(X_{\chi},\CC^*)=d(\chi)$. 
Each map hence descends to an orbifold pencil $\cX \rightarrow \CC_{2,2}$. 
Since one has the commutative diagram: 
 \begin{equation} 
\begin{tikzpicture}[description/.style={fill=white,inner sep=2pt},baseline=(current bounding box.center)] 
\matrix (m) [matrix of math nodes, row sep=3em, 
column sep=2.5em, text height=1.5ex, text depth=0.25ex] 
{ X_\chi\setminus \pi_{\chi}^{-1}(F)& & \CC^*\setminus \{\pm 1\} \\ 
\cal X\setminus F&& \CC_{2,2}\setminus \pi_{\rho}(\pm 1) \\ }; 
\path[->,>=angle 90](m-1-1) edge node[auto] {$\Phi $} (m-1-3); 
\path[->,>=angle 90](m-1-1)edge node[auto,swap] {$i^*(\pi_G) $}(m-2-1); 
\path[->,>=angle 90](m-1-3)edge node[auto] {$\pi_{\rho}$}(m-2-3); 
\path[->,>=angle 90](m-2-1)edge node[auto] {$\phi$}(m-2-3); 
\end{tikzpicture} 
\end{equation} 
where $\pi_{\chi}$ (resp. $\pi_{\rho}$) is the projection $X_{\chi} \rightarrow X$ 
(resp. $\CC^*\rightarrow \CC_{2,2}$).
This diagram induces the isomorphism of order two quotients: 
\begin{equation} 
\pi_1(\cX\setminus F)/\pi_1(X_\chi\setminus \pi_{\chi}^{-1}(F))= 
\pi_1(\CC_{2,2}\setminus \pi_{\rho}(\pm 1))/\pi_1( \CC^*\setminus \{\pm 1\}) 
\end{equation}
which shows that the pencils $\cX \rightarrow \CC_{2,2}$ preserve markings. 
 
Finally, we will check that any such pencil $\Phi:X_\chi \to \CC^*$ can be assumed
to have connected fibers and hence the induced morphism $\Phi_*$ on cohomology is 
surjective, which will imply that the $n$ pencils can be found to be strongly independent. 
Consider the induced orbifold pencil $\phi:\bar \cX \rightarrow \PP^1_{2,2}$ and its Stein
factorization $\bar \cX \rightmap{\tilde \phi} S \rightmap{\tilde \sigma} \PP^1_{2,2}$. Since
$\bar \cX$ is a rational surface, one has $S=\PP^1$ with an orbifold 
structure containing at least 
two orbifold points each locally being quotient 
by an order 2 automorphisms.
 The double cover of $S$ ramified along these two orbifold 
points after removing the preimage of the point at infinity by $\tilde \sigma$ induces
maps $X_\chi \rightmap{\Phi'} \CC^* \rightmap{\sigma} \CC^*$ where $\Phi=\sigma \circ \Phi'$ 
and $\Phi'$ has connected fibers. 
 
The \emph{moreover} part is a direct consequence of Proposition~\ref{prop-alex-roots}. 
\end{proof}
 
\section{Pencils on the complements to plane curves
and zero-dimensional components of characteristic varieties}
\label{curvessection}

\subsection{Essential coordinate components and weight}\label{essentialsection} 
\mbox{} 
 
Essential coordinate components were discovered in~\cite{ArtalCogolludo}.
The dimension of the essential coordinate component is zero since 
the pencil $f: \CC^2\setminus D \rightarrow C$ corresponding to a positive 
dimensional coordinate component of $V_i$, can be extended to a map 
$\bar f:\CC^2\setminus D' \rightarrow C$ with $D'\subsetneq C$ 
(cf.~\cite{charvar}). It follows from~\cite{manuscripta} that essential 
coordinate characters have a finite order. 
Now we can give a Hodge-theoretical characterization of such
in terms of weights (cf. Definition~\ref{weightcharacter}).

\begin{theorem} 
\label{thm-hodge} 
Let $X=\CC^2\setminus D$ be the complement to a plane curve $D$.
A character in the characteristic variety of $\pi_1(X)$ is essential and coordinate 
iff it has weight two. 
\end{theorem}

\begin{proof}
Let $\chi$ be an essential coordinate character and $n$ be its order. 
Denote by $X_n$ the covering space of the complement to the curve
corresponding to the surjection $\pi_1(X) \rightarrow H_1(X,\ZZ/n\ZZ)$. 
Let $\bar X_n$ be a smooth model of the compactification of $X_n$.
Then $W_1 H^1(X_n,\CC)=\im H^1(\bar X_n,\CC)$ (cf.~\cite{deligne}).
If $v \in W_1 H^1(X_n,\CC)$ is a $\chi$-eigenvector of $H_1(X,\ZZ/n)$ 
then $\chi$ is the eigencharacter of the action on the cohomology of the branched
cover. Since $\chi$ is a coordinate character, it follows from Sakuma's formula
(cf.~Theorem~\ref{sakumastatement}) that $\chi$ belongs to the characteristic
variety of the curve $D'$ with components corresponding to the trivial coordinate
of $\chi$. Using Proposition~\ref{weightonanycover}, we obtain that
essential coordinate characters have weight two.

Conversely, if a character has weight two, then it must be coordinate since by
Theorem~\ref{sakumastatement} non-coordinate characters belong to the image of
$H^1(\bar X_n,\CC) \rightarrow H^1(X_n,\CC)$ and hence have weight one. 
This is essential since otherwise Theorem~\ref{sakumastatement} would imply that it
appears as the eigencharacter of a weight one subspace. 
\end{proof}
 
\subsection{2-torsion characters and quasitoric relations} 
\label{sec-qt} 
\mbox{} 
 
As a consequence of Theorem~\ref{thm-main}.\ref{thm-main-part2} one has the following 
interpretation of depth for coordinate 2-torsion characters in terms of quasitoric 
relations of type $(2,2,0)$ (cf.~\cite{mordweil} for a detailed treatment of quasitoric 
relations of elliptic type). 
 
Let $D\subset \PP^2$ be a plane curve and let $\chi$ be a 2-torsion
character of $\pi_1(\PP^2-D)$ having two as the only weight
(i.e. as in Theorem~\ref{thm-hodge}). Let $S$ be the collection of the irreducible divisors
of $D$ and let $G(S)$ the subgroup of the group of
divisors of $\PP^2$ generated by~$S$.
 
Let us fix a generic line at infinity and
identify the coordinate ring of the affine plane with $\CC[x,y]$. 
An element in $\CC(x,y)$ is an $S$-unit (cf.\cite{rosen-s-units}) 
if it is the quotient 
of two polynomials such that the irreducible components of their zero locus 
belong to~$S$. This is a multiplicative group denoted by $E(S)$. 
Note that one has the identification $E(S)/\CC^*\cong G(S)$.
 
An $S$-unit is called \emph{primitive} if
it is a square-free polynomial. The set~$S$ splits into two subsets $S=S_0\cup S_1$ depending on whether or not
$\chi$ ramifies along each divisor, namely, $\chi(\gamma_0)=1$ (resp. $\chi(\gamma_1)=-1$)
for the boundary $\gamma_0$ (resp. $\gamma_1$) of a small disk transversal to any divisor in
$S_0$ (resp. $S_1$).
Let $D_1=\Sigma_{D_{1,i}\in S_1} D_{1,i}$ and $D_0=\Sigma_{D_{0,i}\in S_0} D_{0,i}$.
Note that $D_1$ has necessarily even degree. In other words, $D$ admits an equation $FH=0$,
where $(F)=D_1$ is a polynomial of even degree and $(H)=D_0$. 
 
We shall use the following notations: 
\begin{itemize} 
\item $\CC[x,y]_{E(S)}$ is the localization of $\CC[x,y]$ at $E(S)$, the group of $S$-units.
Hence $E(S)$ is the group of units of $\CC[x,y]_{E(S)}$. 
\item $E(S)^+=E(S)\cap\CC[x,y]$ i.e.
$E(S)^+$ is the multiplicative monoid generated by $\CC^*$ and the polynomials in $E(S)$. 
\item $\KK_F:=\CC(x,y)[\sqrt{F}]$ is the quadratic extension of $\CC(x,y)$. 
\item $\KK_S:=\CC(x,y)[\sqrt{S}]$, the abelian extension of
$\CC(x,y)$ generated by those rational multivalued functions whose
square has an associated divisor which is a formal sum in $S$. 
\item $E(\sqrt{S})$ is the multiplicative group generated by $\CC^*$ and $\sqrt{S}$ and 
$E(\sqrt{S})^+$ is its associated monoid. 
\item $E_P(\sqrt{S})$ is the set of primitive elements of $E(\sqrt{S})$, that is, 
$E_P(\sqrt{S})^+:=\{\alpha\in E(\sqrt{S})\mid\alpha^2\in E(S)^+ \text{ is square-free}\}$. 
\end{itemize} 
We have the following inclusions: 
$$ 
\array{ccccl} 
\CC[x,y] &\subset & \CC[x,y]_{E(S)} & \subset \CC(x,y) \subset \KK_F \subset&\KK_S\\
\rotatebox{90}{$\subset$}& & \rotatebox{90}{$\subset$} &&\rotatebox{90}{$\subset$}\\ 
E(S)^+ &\subset&E(S) & &E(\sqrt{S})  \supset E(\sqrt{S})^+ \supset E_P(\sqrt{S})^+. 
\endarray 
$$ 
We consider the following set: 
$$ 
G_S:=\{(\bar u,\bar v)\in(\KK_S)^2\mid \bar u^2-\bar v^2=1\}, 
$$ 
with a group structure given by: 
\begin{equation} 
\label{eq-suma}
(\bar u_1,\bar v_1)\cdot (\bar u_2,\bar v_2):=(\bar u_1 \bar u_2+\bar v_1 \bar v_2,\bar u_1 \bar v_2+\bar v_1 \bar u_2). 
\end{equation} 
 
Note that $G_S$ is isomorphic to $\KK_S^*$ via the following map
\begin{equation} 
\label{eq-pi} 
\array{cccl} 
\pi : & \KK_S^* & \rightarrow & G_S \\ 
& t & \mapsto & \left(\frac{t+t^{-1}}{2},\frac{t-t^{-1}}{2}\right). 
\endarray 
\end{equation} 
Also, the map $\pi$ is equivariant with respect to the automorphisms 
 $t\mapsto t^{-1}$ of $\KK_S^*$
and $(u,v)\mapsto (u,-v)$ of $G_S$. 
 
\begin{dfn}\label{def-fpair} 
Let $(\bar u,\bar v)\in G_S$. We say that $(\bar u,\bar v)$ is an \emph{$F$-pair} if 
\begin{enumerate} 
\item\label{def-fpair1} $\bar u,\bar v\in E(\sqrt{S})\cdot \CC[x,y]_{E(S)}$ and 
\item\label{def-fpair2} there exists a decomposition $\bar u=\alpha\cdot u$, $\bar v=\beta\cdot v$,
for some $\alpha,\beta\in E(\sqrt{S})$ and $u,v\in\CC[x,y]_{E(S)}$, such that
$\alpha\cdot\beta\in\sqrt{F}\cdot E(S)$. 
\end{enumerate} 
A decomposition satisfying~\eqref{def-fpair2} is called an \emph{$F$-decomposition} of the $F$-pair. 
\end{dfn} 
 
\begin{remark} 
\label{rem-f-decomp} 
Note that any decomposition $\bar u=\alpha\cdot u$, $\bar v=\beta\cdot v$, for some $\alpha,\beta\in E(\sqrt{S})$
and $u,v\in\CC[x,y]_{E(S)}$ of an $F$-pair is an $F$-decomposition. 
\end{remark} 
 
\begin{lemma} 
\label{lem-def-g} 
The set $G$ of $F$-pairs is a subgroup of $G_S$.
\end{lemma} 
 
\begin{proof} 
Since $(\bar u,\bar v)^{-1}=(\bar u,-\bar v)$ it is enough to prove that the condition of $F$-pair is preserved by the product.
Consider $(\bar u_1,\bar v_1),(\bar u_2,\bar v_2)\in G$ and let $\bar u_i=\alpha_i\cdot u_i$, $\bar v_i=\beta_i\cdot v_i$, 
$i=1,2$, be 
$F$-decompositions of these pairs. Hence: 
$$ 
(\bar u_1,\bar v_1)\cdot(\bar u_2,\bar v_2)=(\alpha_1\alpha_2 u_1 u_2+\beta_1\beta_2 v_1 v_2, 
\alpha_1\beta_2 u_1 v_2+\beta_1\alpha_2 v_1 u_2). 
$$ 
For the first coordinate we have 
$$ 
\alpha_1\alpha_2 u_1 u_2+\beta_1\beta_2 v_1 v_2= 
\alpha_1\alpha_2\left( u_1 u_2+(\alpha_1\beta_1 )(\alpha_2\beta_2)\frac{v_1v_2}{\alpha_1^2\alpha_2^2}\right). 
$$ 
Note that $\alpha_i^2,\alpha_i\cdot\beta_i\in E(S)$, $i=1,2$, and hence we have a decomposition of 
this first coordinate. In a similar way we obtain a decomposition of the second coordinate where 
the first factor is $\alpha_1\beta_2$. 
Since 
$$ 
(\alpha_1\alpha_2)(\alpha_1\beta_2)=\alpha_1^2(\alpha_2\beta_2)\in\sqrt{F}\cdot E(S) 
$$ 
the result follows. 
\end{proof} 
 
\begin{dfn}\label{def-fnormal} 
Let $(\bar u,\bar v)\in G$ be an $F$-pair and let $\bar u=\alpha\cdot u$, $\bar v=\beta\cdot v$, be an $F$-decomposition
of $(\bar u,\bar v)$. This decomposition is said to be \emph{normal} if: 
\begin{enumerate} 
\item\label{def-fnormal1} $\alpha=\frac{\tilde \alpha}{\gamma}$, $\beta=\frac{\tilde \beta}{\gamma}$, where
$\tilde \alpha,\tilde \beta,\gamma\in E_P(\sqrt{S})^+$, $\gcd(\tilde \alpha,\tilde \beta,\gamma)=1$.
\item\label{def-fnormal2} $u=\frac{\tilde u}{w}$, $v=\frac{\tilde v}{w}$,
$\tilde u,\tilde v\in\CC[x,y]$, $w\in E(S)^+$, $\gcd(\tilde u,\tilde v,w)=1$. 
\item\label{def-fnormal3} $(\tilde \alpha \tilde \beta)^2=F$. 
\item\label{def-fnormal4} $\gamma^2$ is a divisor of $H$. 
\end{enumerate} 
\end{dfn} 
 
\begin{remark} 
The group $\CC^*$ acts on the set of $F$-normal decompositions via: 
$$\lambda \cdot (\alpha \cdot u,\beta \cdot v) 
\left( (\lambda \alpha) \cdot \left(\lambda^{-1} u\right),(\lambda^{-1} \beta) \cdot \left(\lambda v\right) \right)$$
This action of $\CC^*$ will be referred to as \emph{proportionality}. 
\end{remark} 
 
The following result shows that a normal $F$-decomposition of an $F$-pair 
is almost determined by the pair. 
 
\begin{lemma} 
Any $F$-pair admits a normal $F$-decomposition. Moreover, such $F$-decomposition is unique (up to proportionality). 
\end{lemma} 
 
\begin{proof} 
The uniqueness part is straightforward. Let us start with an arbitrary $F$-decomposition
$\bar u=\alpha \cdot u$, $\bar v=\beta \cdot v$.
We can express for instance $\alpha=\frac{\tilde{\alpha}}{\gamma}$, where
$\tilde{\alpha},\gamma\in E(\sqrt{S})^+$ have no common factors. Note that an element in $E(\sqrt{S})^+$
is not primitive if an only if it contains a factor in $E(S)^+$.
If $\tilde{\alpha}$ is not primitive, 
then $\tilde{\alpha}=\tilde{\alpha}_1 u_1$, where $u_1\in E(S)^+$. In that case, one can rewrite
$\bar u=(\frac{\tilde{\alpha}_1}{\tilde{\gamma}}) (u_1 u)$.
One can assume that both decompositions 
$\bar u=\alpha \cdot u$ and $\bar v=\beta \cdot v$ are such that $\alpha$ and $\beta$ have primitive 
numerators and denominators. Then after taking a common denominator for $\alpha$ and $\beta$ 
(resp. for $u$ and $v$), we can assume that the decomposition
\begin{equation} 
\label{eq-f-decomposition} 
\array{l} 
\bar u=\alpha \cdot u=\left(\frac{\tilde{\alpha}}{\gamma}\right) \cdot \left( \frac{\tilde u}{w}\right),\\ 
\bar v=\beta \cdot v=\left(\frac{\tilde{\beta}}{\gamma}\right) \cdot \left( \frac{\tilde u}{w}\right) 
\endarray 
\end{equation} 
satisfies~\eqref{def-fnormal1} and~\eqref{def-fnormal2}. 
 
By Remark~\ref{rem-f-decomp}, the decomposition $(\alpha \cdot u,\beta \cdot v)$ is an $F$-decomposition and 
hence $\alpha\beta=\frac{\tilde{\alpha}\tilde{\beta}}{\gamma^2}\in\sqrt{F}\cdot E(S)$. 
Let $\sigma\in E_P(\sqrt{S})$ denote an irreducible element. Let us denote by $m(\sigma,\alpha)$ the multiplicity of
$\sigma$ in the decomposition of $\alpha$ into irreducible factors.
Since $\tilde{\alpha},\tilde{\beta},\gamma\in E_P(\sqrt{S})^+$ one concludes 
 that $m(\sigma,\bullet)\in \{0,1\}$ 
for $\bullet=\tilde \alpha,\tilde \beta,\gamma$. 
Moreover, by condition~\eqref{def-fpair2} in Definition~\ref{def-fpair}, if $\sigma^2|F$, then
$m(\sigma,\tilde \alpha)+m(\sigma,\tilde \beta)-2m(\sigma,\gamma)$ is odd and hence so is 
$m(\sigma,\tilde \alpha)+m(\sigma,\tilde \beta)$. The previous two conditions imply that: 
\begin{itemize} 
 \item $m(\sigma,\tilde \alpha)+m(\sigma,\tilde \beta)-2m(\sigma,\gamma)=1$,
 \item $m(\sigma,\gamma)=0$, and thus 
 \item $m(\sigma,\tilde \alpha)+m(\sigma,\tilde \beta)=1.$ 
\end{itemize} 
Hence the second property implies condition~\eqref{def-fnormal4}. 
 
Similarly, if $\sigma^2|H$, then condition~\eqref{def-fpair2} in Definition~\ref{def-fpair}, implies that 
$m(\sigma,\tilde \alpha)+m(\sigma,\tilde \beta)-2m(\sigma,\gamma)$ is even and hence so is 
$m(\sigma,\tilde \alpha)+m(\sigma,\tilde \beta)$. As above, this and the fact that
$m(\sigma,\bullet)\in \{0,1\}$ for $\bullet=\tilde \alpha,\tilde \beta,\gamma$ imply that: 
\begin{itemize} 
 \item $m(\sigma,\tilde \alpha)=m(\sigma,\tilde \beta)=0$ and $m(\sigma,\gamma)\in \{0,1\}$, or 
 \item $m(\sigma,\tilde \alpha)=m(\sigma,\tilde \beta)=1$ and $m(\sigma,\gamma)=0$. 
\end{itemize} 
In order to show condition~\eqref{def-fnormal3} it is enough to prove that the last case can be avoided. 
In order to do so we rewrite $\bar u$ and $\bar v$ as 
\begin{equation} 
\label{eq-f-decomposition2} 
\array{l} 
\bar u=\left(\frac{\frac{\tilde \alpha}{\sigma}}{\sigma \gamma}\right) \cdot
\left(\frac{\sigma^2 \tilde u}{w}\right)=\alpha_1\cdot u_1,\\ 
\bar v=\left(\frac{\frac{\tilde \beta}{\sigma}}{\sigma \gamma}\right) \cdot
\left(\frac{\sigma^2 \tilde v}{w}\right)=\beta_1\cdot v_1, 
\endarray 
\end{equation} 
where $\alpha_1:=\frac{\tilde \alpha_1}{\gamma_1}$, $\beta_1:=\frac{\tilde \beta_1}{\gamma_1}$, 
$\tilde \alpha_1:=\frac{\tilde \alpha}{\sigma}$, $\tilde \beta_1:=\frac{\tilde \beta}{\sigma}$, 
$\gamma_1:=(\sigma\gamma)$, $u_1:=\frac{\tilde u_1}{w}$, $v_1:=\frac{\tilde v_1}{w}$,
$\tilde u_1:=\sigma^2 \tilde u$, and $\tilde v_1:=\sigma^2 \tilde v$. 
Therefore, the second case can be avoided. After a finite number of steps one can assume 
that $\bar u=\alpha_1\cdot u_1$, $\bar v=\beta_1\cdot v_1$ satisfies: 
\begin{enumerate} 
 \item $m(\sigma,\tilde \alpha_1)+m(\sigma,\tilde \beta_1)=1$ for all $\sigma$ such that $\sigma^2|F$, 
 \item $m(\sigma,\tilde \alpha_1)=m(\sigma,\tilde \beta_1)=0$ and $m(\sigma,\gamma_1)\in \{0,1\}$, 
 for all $\sigma$ such that $\sigma^2|H$. 
\end{enumerate} 
Hence it satisfies~\eqref{def-fnormal3}. 
Finally, based on the construction, it is easy to check that this new decomposition also 
satisfies~\eqref{def-fnormal1},\eqref{def-fnormal2}, and~\eqref{def-fnormal4}, that is: 
\begin{enumerate} 
\item $\tilde \alpha_1,\tilde \beta_1\in E_P(\sqrt{S})^+$, $\gcd(\tilde \alpha_1,\tilde \beta_1,\gamma_1)=1$, 
\item $\tilde u_1,\tilde v_1\in\CC[x,y]$, $w\in E(S)^+$, $\gcd(\tilde u_1,\tilde v_1,w)=1$, and 
\item[(4)] $\gamma_1^2$ is a divisor of $H$. 
\hfill $\qedhere$ 
\end{enumerate} 
\end{proof} 
 
Consider a quintuple $(f,g,h,U,V)$ of polynomials in $\CC[x,y]$ satisfying the functional equation
\begin{equation} 
\label{eq-qt-200} 
f U^2-g V^2=h,
\end{equation} 
where
$f\cdot g=F$, $h_{\red}$ (a generator of the radical of $(h)$) divides $H$ and $F,H$ are defined as at the
beginning of this section. Note that $(\CC^*)^2$ acts on the set of such quintuples as follows: given
$\lambda,\mu \in \CC^*$, then
$$(\tilde f,\tilde g,\tilde h,\tilde U,\tilde V)= 
\left(\left(\lambda^2 f\right),\left(\lambda^{-2} g\right), 
\left(\mu^2 h\right),\left(\mu \lambda^{-1} U\right),\left(\mu \lambda V\right)\right).$$
As in the case of $F$-pairs we will denote such action as \emph{proportionality}. 
 
\begin{dfn} 
\label{dfn-qt220} 
A \emph{$(2,2,0)$-quasitoric relation} associated to the character~$\chi$ is
a proportionality class (in the above sense) of quintuple 
$(f,g,h,U,V)$ of polynomials in $\CC[x,y]$
satisfying the functional equation $f U^2-g V^2=h$, where
$f\cdot g=F$, $h_{\text{red}}$ divides $H$ and $F,H$ are defined as at the
beginning of this section. 
\end{dfn} 
 
\begin{remark} 
\label{rem-qt-action} 
As mentioned before Definition~\ref{def-fpair}, the set of $(2,2,0)$-quasitoric relations has a natural order 
two action defined as $(f,g,h,U,V)\mapsto (f,g,h,U,-V)$. Note that $(f,g,h,-U,V)$ is proportional to $(f,g,h,U,-V)$ 
and $(f,g,h,-U,-V)$ is proportional to $(f,g,h,U,V)$. 
\end{remark} 
 
Our purpose now is to establish an isomorphism between $(2,2,0)$-quasitoric relations and 
normal $F$-decompositions.
 
Denote by $Q_{(D,\chi)}$ the set of $(2,2,0)$-quasitoric relations of $D$ associated to the character~$\chi$, that is, 
\begin{equation}
\label{eq-qt-rel}
Q_{(D,\chi)}:=\{(f,g,h,U,V)\in \CC[x,y]^5 \mid f U^2-g V^2=h, f\cdot g=F, \text{\ and\ } h_{\text{red}}|H\}/\sim .
\end{equation}
 
\begin{prop} 
\label{prop-qt-f} 
The set of $(2,2,0)$-quasitoric relations $Q_{(D,\chi)}$ has a group structure 
isomorphic to the group $G$ described above (Lemma{\rm~\ref{lem-def-g}}), where the isomorphism is equivariant with 
respect to the order two actions on both groups. 
\end{prop} 
 
\begin{proof} 
Let $(\bar u,\bar v)$ be an $F$-pair. We are going to associate to $(\bar u,\bar v)$ a $(2,2,0)$-quasitoric relation 
$QT(\bar u,\bar v)$. To this end,
consider a normal $F$-decomposition,
$u=\alpha\frac{U}{W}$, $v=\beta\frac{V}{W}$,
where $\alpha:=\frac{\tilde \alpha}{\gamma}$ and $\beta:=\frac{\tilde \beta}{\gamma}$. Let
$f:=\tilde \alpha^2$, $g:=\tilde \beta^2$, $h_0:=\gamma^2$. From $\bar u^2-\bar v^2=1$, we deduce: 
$$ 
f U^2-g V^2=h_0 W^2. 
$$ 
It is enough to show that $W_{\text{red}}$ divides $H$. We already know that it divides $F H$, hence 
it is enough to show that no irreducible component of $F$ divides $W$. Otherwise it should divide either $f$ or $g$, 
say $f$ for simplicity, then it divides $V$ (and not $g$). Therefore its multiplicity is odd in $f U^2$
and even in $g V^2$ as well as in $h_0 W^2$, which is a contradiction. If $h:=h_0 W^2$, then 
$QT(\bar u,\bar v):=(f,g,h,U,V)$ is a $(2,2,0)$-quasitoric relation of~$\chi$. 
 
Notice that up to proportionality, one has $u=(\lambda \alpha)\frac{\lambda^{-1}\mu U}{\mu W}$,
$v=(\lambda^{-1} \beta)\frac{\lambda \mu V}{\mu W}$, where
$\lambda \alpha:=\frac{\lambda \tilde \alpha}{\gamma}$ and
$\lambda^{-1} \beta:=\frac{\lambda^{-1} \tilde \beta}{\gamma}$. Define 
$\tilde f:=\left(\lambda^2 f\right)$, 
$\tilde g:=\left(\lambda^{-2} g\right)$, 
$\tilde W:=\left(\mu^2 W^2\right)$, 
$\tilde U:=\left(\mu \lambda^{-1} U\right)$, and
$\tilde V:=\left(\mu \lambda V\right)$. From $\bar u^2-\bar v^2=1$, we deduce: 
$$ 
\tilde f \tilde U^2-\tilde g \tilde V^2=h_0 \tilde W^2, 
$$ 
and hence $(\tilde f,\tilde g,h_0 \tilde W^2,\tilde U,\tilde V)\sim (f,g,h,U,V)$. 
 
Conversely, let us fix a $(2,2,0)$-quasitoric relation $QT=(f,g,h,U,V)$.
Let $\tilde \alpha:=\sqrt{f}$ and $\tilde \beta:=\sqrt{g}$ 
and write down $h:=h_0 W^2$ where $h_0$ is square-free. Denote $\gamma:=\sqrt{h_0}$. Then 
$$ 
\left(\frac{\tilde \alpha}{\gamma}\frac{U}{W},\frac{\tilde \beta}{\gamma}\frac{V}{W}\right) 
$$ 
is the normal $F$-decomposition of the $F$-pair $(u,v)$ such that $QT(u,v)=QT$.
 
Finally, note that if $(\bar u,\bar v)$ is the normal $F$-decomposition associated with the
$(2,2,0)$-quasitoric relation $(f,g,h,U,V)$ as at the beginning of this proof, then
$(\bar u,-\bar v)$ will be associated with $(f,g,h,U,-V)$, and hence the result follows. 
\end{proof} 
 
Consider $X_2$ the double cover of $X=\CC^2\setminus D$ associated to the order two character~$\chi$. 
 
\begin{prop} 
\label{prop-qt-iso} 
The group $Q_{(D,\chi)}$ is isomorphic to $\Mor_{\ZZ_2}(X_2,\CC^*)$. 
\end{prop} 
 
\begin{proof} 
We will give a constructive proof of this result. Suppose one has a quasitoric relation
$(f,g,h,U,V)\in Q_{(D,\chi)}$. There is a rational map $\CC^2\rightarrow \PP^1$ defined by
$(x,y)\mapsto [fU^2:gV^2]$ such that $f U^2 - g V^2 = h$. 
This map, restricted to $\tilde X=\{fgh=0\}$ defines an orbifold morphism
$\tilde X\to \CC_{2,2}=\PP^1_{(2,[0:1]),(2,[1:0])}\setminus \{[1:1]\}$. Finally, since $\{fg=0\}=\{F=0\}$ and
$\{h=0\}\subset \{H=0\}$ by definition, one can restrict this to a well-defined orbifold morphism
$(X,\chi)\to (\CC_{2,2},\rho)$, which induces an equivariant morphism in~$\Mor_{\ZZ_2}(X_2,\CC^*)$. 
 
Conversely, any equivariant morphism in $\Mor_{\ZZ_2}(X_2,\CC^*)$ induces a morphism of marked orbifolds
on the quotient $(X,\chi)\to (\CC_{2,2},\rho)$. Extending this to $\PP^2$ one obtains a rational morphism
$\PP^2 \dashrightarrow \CC_{2,2}=\PP^1_{(2,[0:1]),(2,[1:0])}\setminus \{[1:1]\}$, which on a generic affine
chart can be defined by $(x,y)\mapsto [fU^2:gV^2]$, where $\{fg=0\}$ corresponds to the ramified part of~$D$ 
and $fU^2-gV^2=h$ with $\{\tilde h=0\}\subset \{H=0\}$ the unramified part of~$D$. This quasitoric relation
defines an element of~$Q_{(D,\chi)}$ and the order-two action corresponds with the covering transformations. 
 
Finally, let us check that this bijection is in fact a homomorphism. Note that an element
$p:=(f,g,h,U,V)\in Q_{(D,\chi)}$ produces the following commutative diagram: 
 
\begin{equation}\label{eq-comm-diag-x21} 
\begin{tikzpicture}[description/.style={fill=white,inner sep=2pt},baseline=(current bounding box.center)] 
\matrix (m) [matrix of math nodes,
column sep=2.5em, text height=1.5ex, text depth=0.25ex] 
{{[x:y:z:w]} & {[\alpha u:\beta v]} \\[1em] 
X_2 & \PP^1\setminus \{[1:1],[1:-1]\}\\[3em] 
X &\PP^1_{(2,[1:0]),(2,[0:1])}\setminus \{[1:1]\}\\[1em] 
{[x:y:z]} &{[f U^2: g V^2]},\\ }; 
\path[|->,shorten >=7.5mm,shorten <=2.5mm,>=angle 90](m-1-1) edge node[auto] {} (m-1-2); 
\path[->,>=angle 90](m-2-1) edge node[auto] {$\Psi_p$} (m-2-2); 
\path[->,>=angle 90](m-3-1) edge node[auto] {$\psi_p$} (m-3-2); 
\path[->,>=angle 90](m-4-1) edge node[auto] {} (m-4-2); 
\path[->,>=angle 90](m-2-1)edge node[auto,swap] {}(m-3-1); 
\path[->,>=angle 90](m-2-2)edge node[auto,swap] {}(m-3-2); 
\end{tikzpicture} 
\end{equation} 
where $X_2$ is contained in $\{[x:y:z:w]\in \PP^3 \mid w^2=fg=F\}$ and $(\alpha u,\beta v)$ is the 
normal $F$-decomposition associated with $p$ according to Proposition~\ref{prop-qt-f}. 
For convenience, let us change the coordinates of $\PP^1$ so that $[1:1]\mapsto [1:0]$ and $[1:-1]\mapsto [0:1]$. 
In that case, $\PP^1\setminus \{[1:1],[1:-1]\}$ becomes $\CC^*=\PP^1\setminus \{[1:0],[0:1]\}$ and the new 
equation of $\Psi_p:X_2\to \CC^*$ becomes $\Psi_p(x,y,z,w)=(\alpha u+\beta v)^2$. Moreover,
diagram~\eqref{eq-comm-diag-x21} becomes 
 
\begin{equation}\label{eq-comm-diag-x22} 
\begin{tikzpicture}[description/.style={fill=white,inner sep=2pt},baseline=(current bounding box.center)] 
\matrix (m) [matrix of math nodes,
column sep=2.5em, text height=1.5ex, text depth=0.25ex] 
{{[x:y:z:w]} & (\alpha u+\beta v)^2 \\[1em] 
X_2 & \CC^*\\[3em] 
X &\CC_{(2,\sqrt{-1}),(2,-\sqrt{-1})}\\[1em] 
{[x:y:z]} &2\alpha\beta uv,\\ }; 
\path[|->,shorten >=7.5mm,shorten <=2.5mm,>=angle 90](m-1-1) edge node[auto] {} (m-1-2); 
\path[->,>=angle 90](m-2-1) edge node[auto] {$\Psi_p$} (m-2-2); 
\path[->,>=angle 90](m-3-1) edge node[auto] {$\psi_p$} (m-3-2); 
\path[->,>=angle 90](m-4-1) edge node[auto] {} (m-4-2); 
\path[->,>=angle 90](m-2-1)edge node[auto,swap] {}(m-3-1); 
\path[->,>=angle 90](m-2-2)edge node[auto,swap] {}(m-3-2); 
\end{tikzpicture} 
\end{equation} 
where the vertical right map is the double cover $t\mapsto -\frac{t-t^{-1}}{2}$ ramified at $t=\pm \sqrt{-1}$ 
with values $\mp \sqrt{-1}$ and where $(\alpha u+\beta v)^{-1}=(\alpha u-\beta v)$. 
 
Given two quasitoric relations $p_i:=(f_i,g_i,h_i,U_i,V_i)\in Q_{(D,\chi)}$, $i=1,2$ where $f_1g_1=f_2g_2=F$ 
and $h_i|H$. Consider $(\bar u_i,\bar v_i)=(\alpha_i u_i,\beta_i v_i)$, $i=1,2$ the normal $F$-decompositions
associated with $p_i$ according to Proposition~\ref{prop-qt-f}. From~\eqref{eq-suma} one can see that 
\begin{equation*} 
\Psi_{p_1p_2}=(\bar u_1\bar u_2+\bar v_1\bar v_2+\bar u_1\bar v_2+\bar u_2\bar v_1)^2= 
(\bar u_1 + \bar v_1)^2(\bar u_2+\bar v_2)^2=\Psi_{p_1}\Psi_{p_2}. 
\qedhere 
\end{equation*} 
\end{proof} 
 
As a consequence, under the conditions of Theorem~\ref{thm-main}
one obtains Theorem~\ref{thm-qt220}. 
 
\begin{proof}[Proof of Theorem{\rm~\ref{thm-qt220}}] 
The result follows from Proposition~\ref{prop-qt-iso} and the fact that
$$ 
\Mor_{\ZZ_2}(X_2,\CC^*)=\Hom_{\ZZ_2}(\Alb(X_2)^-,\CC^*)= 
\Mor_{\ZZ_2}((\CC^*)^d,\CC^*)=\Hom(\ZZ^d,\ZZ)=\ZZ^d, 
$$ 
see~\eqref{eq-alb}, where $d\leq d(\chi)$ and the equality holds if the character only has weight~2. 
\end{proof} 
 
\section{Examples} \label{sec-examples} We refer to~\cite{pisa} for explicit examples 
of cases of characters of depth greater than one and corresponding
orbifold pencils in the case of complements to plane reducible curves. 

In this section we will present two examples illustrating interesting phenomena about 2-torsion 
characters that have come up during the preparation of this paper and that might be of interest 
to the reader.

\subsection{Orbifold pencils of type $(2,2,2,2)$}
Note that Theorem~\ref{thm-main}\ref{thm-main-part2} refers only to weight two
characters of order two. Characters of order two and weight one might be associated 
with elliptic orbifold pencils of type $(2,2,2,2)$ as the following example seems 
to suggest. 

Consider the Hesse arrangement $\cH$ of lines, given by the twelve lines
$\{\ell_1,\dots,\ell_{12}\}$ joining the inflection points of a smooth plane cubic. It
is easy to check that the system of cubics sharing the nine inflection points is a pencil
with exactly four singular fibers. Each one of these fibers is a completely reducible
curve given by three lines in general position. The four cubics
$\cC_k:=\{\ell_{3k+1}\ell_{3k+2}\ell_{3k+3}=0\}$, $k=0,1,2,3$ belong to a pencil and their union 
gives the Hesse arrangement. After blowing up the base points of this pencil one obtains an elliptic 
fibration onto $\PP^1$ where $P_k\in \PP^1$ is the image of the special fiber $\cC_k$, $k=0,1,2,3$.
If one further blows up one of the three double points in each special fiber, one obtains a rational
surface~$\tilde \PP^2$, four exceptional vertical divisors $E_0,E_1,E_2,E_3$ (not sections) and twelve
strict transforms $\tilde \ell_i$, $i=1,\dots,12$. The surface
$X:=\tilde \PP^2\setminus \cup_{i=1}^{12} \tilde \ell_i$ together with the elliptic fibration induces
a well-defined orbifold map onto $\PP^1_{(2,P_0),(2,P_1),(2,P_2),(2,P_3)}$ since the preimage of $P_k$
in $X$ is given by $2E_k$ (in divisor notation). 

\subsection{Ceva Arrangements}
Note that the polynomials $f$ and $g$ in the group of quasitoric relations~\ref{eq-qt-rel} must 
satisfy $fg=F$ but such partition of $F$ might be different for different quasitoric relations as the
following example shows. Consider the following set of lines:
\begin{equation}
\array{c}
\ell_1:=x, \quad
\ell_2:=y, \quad
\ell_3:=z, \quad
\ell_4:=(y-z),\\
\ell_5:=(x-z), \quad
\ell_6:=(x-y), \quad
\ell_7:=(x-y-z).
\endarray
\end{equation}
The curve $D:=\left\{ \prod_{i=1}^7\ell_i=0 \right\}$ is a realization of the \SC\ arrangement $\Ceva(2,1)$ 
(cf.~\cite[Section~2.3.J, pg.~81]{geraden}) otherwise known as the non-Fano plane. In~\cite{pisa}, a computation
of the 2-torsion characters of $D$ is presented via orbifold pencils. In particular, consider 
$G:=\pi_1(\PP^2\setminus D)$ (whose abelianization is 
$\frac{\ZZ\gamma_1 \oplus \dots \oplus \ZZ\gamma_7}{\gamma_1+\dots +\gamma_7}$). A character on $G$ can be 
represented by a septuple of complex numbers whose product is 1, the $i$-th coordinate representing the image 
of any meridian $\gamma_i$ around~$\ell_i$. The element $\chi=(1,-1,-1,1,-1,-1,1)$ represents hence a character
on~$G$. In fact, it is well known that its depth is two. Note that $F=\ell_2\ell_3\ell_5\ell_6$ whereas 
$H=\ell_1\ell_4\ell_7$ according to the notation introduced in section~\ref{sec-qt}. Note that 
\begin{equation}
\array{c}
\ell_2\ell_5-\ell_3\ell_6=\ell_1\ell_4,\\
\ell_2\ell_6-\ell_3\ell_5=\ell_4\ell_7
\endarray
\end{equation}
are quasitoric relations of type $(2,2,0)$ corresponding to the quintuples 
$q_1:=(\ell_2\ell_5,\ell_3\ell_6,\ell_1\ell_4,1,1)$ and 
$q_2:=(\ell_2\ell_6,\ell_3\ell_5,\ell_4\ell_7,1,1)$ respectively. It is not hard to show that $q_1$ and
$q_2$ are strongly independent and they generate~$\cQ_{(D,\chi)}$.

Finally, note that 
$$
\ell_2\ell_5\ell_8^2-\ell_3\ell_6\ell_9^2=\ell_1\ell_4\ell_7^2,
$$
where $\ell_8:=(y-z-x)$ and $\ell_9:=(z-x-y)$ is another quasitoric decomposition corresponding to 
$q_3=(\ell_2\ell_5,\ell_3\ell_6,\ell_1\ell_4\ell_7^2,\ell_8,\ell_9)\in \cQ_{(D,\chi)}$. 
One can check that in fact~$q_3=-q_1+2q_2$ (in additive notation).


\begin{thebibliography}{99} 

\bibitem{abel}
N.~H.~Abel, \emph{Sur l'int\'egration de la formule diff\'erentielle $\rho dx/\sqrt{R}$, $R$ et $\rho$
\'etant des fonctions enti\`eres}. Oeuvres Compl\`etes de Niels Henrik Abel (L. Sylow and S. Lie, eds.). 
Christiania, t.~\textbf{1}, 104--144. 1881.


\bibitem{adem} A.~Adem, J.~Leida, and Y.~Ruan,
\emph{Orbifolds and stringy topology}. Cambridge University Press. 2007. 
 
\bibitem{arapura} 
D.~Arapura, \emph{Geometry of cohomology support loci for local systems {I}},
J. of Alg. Geom. \textbf{6} (1997), 563--597.

\bibitem{ArtalCogolludo}
E.~Artal, J.~Carmona, and J.I.~Cogolludo, 
\emph{Essential coordinate components of characteristic varieties},
Math. Proc. Cambridge Philos. Soc. \textbf{136} (2004), no.~2, 287--299.
 
\bibitem{kike-hiro-survey} 
E.~Artal, J.I.~Cogolludo, and H.O~Tokunaga,
\emph{A survey on {Z}ariski pairs}, Algebraic geometry in {E}ast
{A}sia---{H}anoi 2005, Adv. Stud. Pure Math., vol.~50, Math. Soc. Japan, 
Tokyo, 2008, pp.~1--100. 

\bibitem{pisa} 
E. Artal, J.I.~Cogolludo-Agust{\'\i}n and A.~Libgober, 
\emph{Characters of fundamental groups of curve complements and orbifold pencils},
Preprint available at \texttt{arXiv:1108.0164 [math.AG]}, 2011. 

\bibitem{acm-arapura}
E. Artal, J.I.~Cogolludo-Agust{\'\i}n and D.~Matei, 
\emph{Characteristic varieties of quasi-projective manifolds and orbifolds}, 
Preprint available at \texttt{arXiv:1005.4761 [math.AG]}, 2010.

\bibitem{acm-hefei}
\bysame, 
\emph{Orbifold groups, quasi-projectivity and covers},
Preprint available at \texttt{arXiv:1203.1645 [math.AG]}, 2012.

\bibitem{geraden}
G.~Barthel, F.~Hirzebruch, and T.~H{\"o}fer, \emph{Geradenkonfigurationen und
{A}lgebraische {F}l\"achen}, Friedr. Vieweg \& Sohn, Braunschweig, 1987.

\bibitem{mordweil} 
J.I.~Cogolludo-Agust{\'{\i}}n, A.~Libgober, 
\emph{Mordell-Weil groups of elliptic threefolds and the Alexander module of plane curves}, 
Preprint available at \texttt{arXiv:1008.2018v2 [math.AG]}, 2010.

 

\bibitem{deligne} P.~Deligne, 
\emph{\'{E}quations diff\'erentielles \`a points singuliers r\'eguliers},
Lecture Notes in Mathematics, Vol. 163, Springer-Verlag, Berlin, 1970. 
 
\bibitem{delignehodge} \bysame, 
\emph{Th\'eorie de Hodge. {II}, {III}},
Inst. Hautes \'Etudes Sci. Publ. Math. \textbf{40}, (1971), 5--57. 
\textbf{44} (1974). 
 
\bibitem{dimca-pencils}
A.~Dimca, \emph{Pencils of plane curves and characteristic varieties}
Preprint available at \texttt{math.AG/0606442}, 2006.


\bibitem{friedman}
R.~Friedman and J.~W.~Morgan, \emph{Smooth four-manifolds and complex
surfaces}, Ergebnisse der Mathematik und ihrer Grenzgebiete (3), vol.~27,
Springer-Verlag, Berlin, 1994.

\bibitem{hazama-pell}
F.~Hazama, \emph{Pell equations for polynomials}, Indag. Math. (N.S.)
  \textbf{8} (1997), no.~3, 387--397.

\bibitem{hazama-twists}
\bysame, \emph{Twists and generalized {Z}olotarev polynomials}, Pacific J.
  Math. \textbf{203} (2002), no.~2, 379--393.

\bibitem{eko} E.~Hironaka, 
\emph{Abelian coverings of the complex projective plane branched 
along configurations of real lines}, Mem. Amer. Math. Soc. \textbf{105} 
(1993), no.~502, vi+85. 
 
\bibitem{iitaka-logarithmic} 
S.~Iitaka, \emph{Logarithmic forms of algebraic varieties}, J. Fac. Sci. Univ. 
Tokyo Sect. IA Math. \textbf{23} (1976), no.~3, 525--544. 
 
\bibitem{abcovhomology} A.~Libgober,
On the homology of finite abelian coverings.
Topology Appl. 43 (1992), no. 2, 157--166. 
 
\bibitem{charvar} 
\bysame, \emph{Characteristic varieties of algebraic curves}, Applications of
algebraic geometry to coding theory, physics and computation (Eilat, 2001),
Kluwer Acad. Publ., Dordrecht, 2001, pp.~215--254.

 
\bibitem{manuscripta} \bysame,
\emph{Non vanishing loci of Hodge numbers of local systems},
Manuscripta Math. \textbf{128} (2009), no.~1, 1--31.
 
\bibitem{catalan} \bysame,
\emph{On combinatorial invariance of the cohomology of Milnor fiber of arrangements
and Catalan equation over function field},
Preprint available at \texttt{arXiv:1011.0191v2 [math.AG]}, 2010.
 


\bibitem{rosen-s-units} 
M.~Rosen, \emph{{$S$}-units and {$S$}-class group in algebraic function 
fields}, J. Algebra \textbf{26} (1973), 98--108. 
 

\bibitem{shabat-zvonkin}
G.~Shabat and A.~Zvonkin, \emph{Plane trees and algebraic numbers}, Jerusalem
  combinatorics '93, Contemp. Math., vol. 178, Amer. Math. Soc., Providence,
  RI, 1994, pp.~233--275.

\bibitem{sakuma} 
M.~Sakuma, \emph{Homology of abelian coverings of links and spatial graphs}, 
Canad. J. Math. \textbf{47} (1995), no.~1, 201--224. 
 
\bibitem{scott}
P.~Scott, \emph{The geometries of {$3$}-manifolds}, Bull. London Math. Soc. 
\textbf{15} (1983), no.~5, 401--487. 
 
\bibitem{serre} J.P.~Serre,
\emph{Morphismes universels et vari\'et\'e d'Albanese}.
Seminaire Chevalley, ann\'ee 1958/59, expos\'e 10. 
 
\bibitem{simpson} 
C.~Simpson, \emph{A weight two phenomenon for the moduli of rank one local 
systems on open varieties}, From {H}odge theory to integrability and {TQFT} 
tt*-geometry, Proc. Sympos. Pure Math., vol.~78, Amer. Math. Soc., 
Providence, RI, 2008, pp.~175--214. 
 

\bibitem{suciu-translated} A.I.~Suciu,
\emph{Translated tori in the characteristic varieties of complex hyperplane arrangements},
Topology Appl. \textbf{118} (2002), no.~1--2, 209--223,
Arrangements in Boston: a Conference on Hyperplane Arrangements (1999).
\end{thebibliography}
\end{document}